%% file: sym.tex
\documentclass{amsart}
\usepackage{graphicx}
\usepackage{mathtools,etoolbox}

\usepackage{ifthen}
\usepackage[T1]{fontenc}
\usepackage[utf8]{inputenc}
\usepackage[all]{xy}
\usepackage{graphicx}
\usepackage{enumerate}
\usepackage{xspace}
\usepackage{pdfsync}
\usepackage{epic}
\usepackage{dsfont}
\usepackage{tikz-cd}
\usepackage{amssymb}
\usepackage{stmaryrd}

\usepackage{multirow}

\usepackage{hyperref}
\usepackage{varioref}
\usepackage[hyperpageref]{backref}

\usepackage{xcolor}
\hypersetup{
    colorlinks,
    linkcolor={red!60!black},
    citecolor={blue!60!black},
    urlcolor={blue!50!black}
}

\theoremstyle{theorem} \newtheorem{theorem}{Theorem}
\newtheorem{proposition}{Proposition} [section]
\newtheorem{lemma}[proposition]{Lemma}
\newtheorem{definition}[proposition]{Definition}
\newtheorem{corollary}[proposition]{Corollary}
\newtheorem{conjecture}[proposition]{Conjecture}
\newtheorem{problem}[proposition]{Problem}
\theoremstyle{definition} \newtheorem{question}{Question}
\theoremstyle{remark} \newtheorem{example}{Example}
\theoremstyle{remark} \newtheorem{remark}[proposition]{Remark}
\theoremstyle{definition} \newtheorem{defi}[proposition]{Definition}
\theoremstyle{definition}

\def\C{\mathbb{C}}
\def\P{\mathbb{P}}

\newcommand{\quotient}[2]{{\left.\raisebox{-.2em}{$#1$}\middle\backslash\raisebox{.2em}{$#2$}\right.}}

\newcommand\Q{{\mathbb{Q}}}

\def\det{\mathop{\rm det}\nolimits}

\def\Supp{\mathop{\rm Supp}\nolimits}
\def\Sym{\mathop{\rm Sym}\nolimits}
\def\Jac{\mathop{\rm Jac}\nolimits}

\newcommand{\calL}{{\mathcal L}}

\newcommand{\calO}{{\mathcal O}}

\newcommand{\PP}{{\mathbb P}}

\def\bydef{\coloneqq}

\begin{document}
\title[]{Hyperbolicity and specialness of symmetric powers}

\

\author{Beno\^it Cadorel}
\address{Institut \'{E}lie Cartan de Lorraine, UMR 7502, Universit\'{e} de Lorraine, Site de Nancy, B.P. 70239, F-54506 Vandoeuvre-l\`{e}s-Nancy Cedex}

\

\email{benoit.cadorel@univ-lorraine.fr}
\author{Fr\'ed\'eric Campana}
\address{Institut \'{E}lie Cartan de Lorraine, UMR 7502, Universit\'{e} de Lorraine, Site de Nancy, B.P. 70239, F-54506 Vandoeuvre-l\`{e}s-Nancy Cedex}

\

\email{frederic.campana@univ-lorraine.fr}

\author{Erwan Rousseau}
\address{Institut Universitaire de France \& Aix Marseille Univ., CNRS, Centrale Marseille, I2M, Marseille, France}

\

\email{erwan.rousseau@univ-amu.fr}

\thanks{E. R. was partially supported by the ANR project \lq\lq FOLIAGE\rq\rq{}, ANR-16-CE40-0008.}

\date{\today}

\begin{abstract}
Inspired by the computation of the Kodaira dimension of symmetric powers $X_m$ of a complex projective variety $X$ of dimension $n \geq 2$ by Arapura and Archava, we study their analytic and algebraic hyperbolic properties. First we show that $X_m$ is special if and only if $X$ is special (except when the core of $X$ is a curve). Then we construct dense entire curves in (sufficiently hig) symmetric powers of K3 surfaces and product of curves. We also give a criterion based on the positivity of jet differentials bundles that implies pseudo-hyperbolicity of symmetric powers. As an application, we obtain the Kobayashi hyperbolicity of symmetric powers of generic projective hypersurfaces of sufficiently high degree. On the algebraic side, we give a criterion implying that subvarieties of codimension $\leq n-2$ of symmetric powers are of general type. This applies in particular to varieties with ample cotangent bundles. Finally, based on a metric approach we study symmetric powers of ball quotients.

\end{abstract}

\maketitle



\input{intro}

\input{notations}

\input{special}

\input{canonical}

\input{entirecurves}

\input{secthyp.tex}

\input{metric.tex}

\input{constants.tex}

\bibliographystyle{alpha}
\bibliography{biblio}

\end{document}

%% file: intro.tex
\section{Introduction}
In \cite{AA02}, it is shown that for any smooth complex projective variety $X$ with $n=\dim X \geq 2$ and Kodaira dimension $k$, the Kodaira dimension of the symmetric product $X_m$ is equal to $mk$. In particular, $X$ is of general type if and only if $X_m$ is of general type. Green-Griffiths-Lang conjectures claim that varieties of general type should have hyperbolic properties concerning entire curves or rational points. More precisely, varieties of general type should be \emph{pseudo-hyperbolic} i.e. there should exist a proper subvariety containing all entire curves and all but finitely many rational points. Therefore, if we believe in these conjectures, symmetric product of varieties of general type should share the same hyperbolic properties and the following conjecture should be true. 

\begin{conjecture}\label{symLang}
Let $X$ be a complex projective variety with $n=\dim X \geq 2$. Then $X$ is pseudo-hyperbolic if and only if $X_m$ is pseudo-hyperbolic for some $m$.
\end{conjecture}

Remark that if $X_m$ is pseudo-hyperbolic for some $m$, it is not difficult to prove that $X$ is pseudo-hyperbolic, so the interesting question is to show that $X_m$ is pseudo-hyperbolic if $X$ is.

The second author has proposed vast generalizations of Green-Griffiths-Lang conjectures based on the \emph{specialness} property. Special varieties are opposite to varieties of general type in the following sense: they do not admit any fibration with (orbifold) base of general type or equivalently the core is of dimension $0$ (see \cite{Cam04} for a detailed background on special varieties and the core map). Conjecturally, they should correspond to varieties admitting Zariski dense entire curves and a (potentially) dense set of rational points (when defined over a number field).

In the first part of this paper we study specialness of symmetric powers and prove the following result.

\begin{theorem}
If $n\geq 2$, $X$ is special if and only if so is $X_m$ for some $m$ (unless when $n\geq 2$, and the core of $X$ is a curve, in which case $X_m$ is special for sufficiently big $m$ even if $X$ is not special).
\end{theorem}

In the line of the above conjectures, one should expect to have the corresponding hyperbolic properties. In fact, this phenomenon has already been studied in \cite{HT00} where the authors prove potential density of rational points for sufficiently high symmetric powers of K3 surfaces. Here we prove the analytic analogue showing that these symmetric powers contain dense entire curves (see Theorem \ref{C^2-dom K3}).

We also study the case of product of curves.

\begin{theorem}
Let $G$ (resp. $C)$ be a curve of genus $g(G)\leq 1$ (resp. $g(C)>1)$, and  $S=G\times C$.
If $m\geq g$, then $S_m$ contains dense entire curves.
\end{theorem}

As recently observed by Levin \cite{L}, such symmetric powers provide negative answers to Puncturing Problems as formulated by Hassett and Tschinkel in \cite{HT} in the arithmetic and geometric setting, which can be generalized in the analytic setting as follows.

\begin{problem}(Analytic Puncturing Problem)
Let $X$ be a projective variety with canonical singularities and $Z$ a subvariety of codimension $\geq 2$. Assume that there are Zariski dense entire curves on $X$. Is there a Zariski dense entire curve on $X \setminus Z$ ?
\end{problem}

Considering $Z:=\Delta_m \subset S_m$, the small diagonal, one easily gets that Zariski dense entire curves cannot avoid $Z$ giving a counter-example to the above problem.

In the second part of this paper, we study hyperbolic properties of symmetric powers. We were not able to prove conjecture \ref{symLang} in full generality, and the following particular case seems already interesting and nontrivial.

\begin{problem}
Let $X$ be a complex projective manifold with $\dim X \geq 2$. Assume $\Omega_X$ is ample. Show that $X_m$ is pseudo-hyperbolic.
\end{problem}

We provide partial answers to this problem and consider more generally jet differentials bundles $E_{k,r}^{GG} \Omega_X$ whose sections correspond to algebraic differential equations or equivalently to sections of lines bundles on the jets spaces $\pi_k: X_k^{GG} \to X$ (see section \ref{jets} and \cite{Dem95} for an introduction to these objects). First, we establish a criteria which ensures strong algebraic degeneracy of entire curves in symmetric powers i.e. the Zariski closure of the union of entire curves, known as the exceptional set $\mathrm{Exc}(X_m)$, is a proper subvariety.

\begin{theorem}
	Let $X$ be a complex projective manifold. Let $A$ be a very ample line bundle on $X$. Let $Z \subsetneq X$, and $k, r ,d \in \mathbb N^\ast$. We make the following hypotheses.

\begin{enumerate}\itemsep=0.5em
		\item Assume that
			$$
			\mathrm{Bs} \left(H^0 (X, E_{k,r}^{GG} \Omega_X \otimes \mathcal O(- d A)) \right) \subset X_k^{GG, \mathrm{sing}} \cup \pi_k^{-1}(Z).
			$$

		\item Assume that $\frac{d}{r} > 2 m(m-1)$.
\end{enumerate}
	Then $\mathrm{Exc}(X_m)\neq X_m$. 
\end{theorem}

In fact, there is a precise description of a proper subvariety containing the exceptional locus (see Theorem \ref{thmhypcrit} for details). Our criteria applies to all situations where we have enough positivity for jet differentials bundles. Thanks to the recent works around the Kobayashi conjecture \cite{bro17, den17, dem18, RY18, BK19}, we obtain the following result.

\begin{theorem}
Let $X \subset \mathbb P^{n+1}$ be a generic hypersurface of degree
	$$
	d \geq (2n - 1)^5 (2m^2 + 10 n - 1).
	$$
	Then $X_m$ is hyperbolic.
\end{theorem}

Then we establish a criterion ensuring that any subvariety $V \subset X_m$ of ${\rm codim} V \leq n - 2$ is of general type (see Theorem \ref{thmsubv}). It applies in particular to varieties with ample cotangent bundle.

\begin{theorem} \label{thmsubv}
	Let $X$ be a complex projective manifold with $n=\dim X \geq 2$. Assume $\Omega_X$ is ample. Then, any subvariety $V \subseteq X_m$ such that ${\rm codim} V \leq n - 2$ and $V \not\subset  X_m^{\rm sing}$ is of general type.
\end{theorem}

In view of the above problem, if we believe in Lang's conjectural characterization of the exceptional locus, this should tell that ${\rm codim\, Exc}(X_m) \geq n - 1$ for varieties with $\Omega_X$ ample.

As a first corollary, we obtain geometric restrictions on the exceptional locus of non-hyperbolic \emph{algebraic} curves in $X_m$.

\begin{corollary} \label{corolexcalg} Assume that $\Omega_X$ is ample. Then, there exist countably many proper algebraic subsets $V_k \subsetneq X_m$ $(k \in \mathbb N)$ containing the image of any non-hyperbolic algebraic curve, such that ${\rm codim}_{X_m} (V_k) \geq n - 1$ for all $k \in \mathbb N$.
\end{corollary}

We also obtain some genus estimates for curves lying on $X$ in the spirit of \cite[Corollary 4]{AA02}.

\begin{corollary}
	Assume that $\Omega_X$ is ample, and let $Y \subset X$ be a closed submanifold. Let $1 \leq l \leq d$ be integers. Assume that $l$ generic points of $Y$ and $d - l$ generic points of $X$ lie on a curve of geometric genus $g$.
	Then if $$l \cdot {\rm codim} Y \leq  \dim X - 2,$$ we have $g > d$.
\end{corollary}

In the last section, we give some metric criteria which in particular apply to quotient of bounded symmetric domains. We obtain the following result for ball quotients.

\begin{theorem} Let $X = \quotient{\Gamma}{\mathbb B^n}$ be a ball quotient by a torsion free lattice with only unipotent parabolic elements, and let $\overline{X} = X \cup D$ be a smooth minimal compactification (see \cite{mok12}). Let $m \geq 1$. Then :
	\begin{enumerate}[(a)]
		\item Let $V \subset \overline{X}_m$ be a subvariety with ${\rm codim} V \leq n - 6$ and $V \not\subset \mathfrak{d}_1(D) \cup (\overline{X}_m)_{\rm sing}$ (where $\mathfrak{d}_i (V) = \{ [x_1, ..., x_m] \in \overline{X}_m | x_1, ..., x_i \in V \}$). Then $V$ is of general type.

		\item Let $p \geq n(m-1) + 6$, and $f : \mathbb C^p \to \overline{X}_m$ be a holomorphic map such that $f(\mathbb C^p) \not\subset  \mathfrak{d}_1(D) \cup (\overline{X}_m)_{\rm sing}$. Then ${\rm Jac}(f)$ is identically degenerate. 
	\end{enumerate}	
\end{theorem}

\subsubsection*{Acknowledgments} 
The authors would like to thank Ariyan Javanpeykar for very interesting discussions about several themes of this paper. They also thank Aaron Levin for sharing with them his paper on Puncturing Problems.

%% file: notations.tex
\section{Notations and conventions} 
We introduce here a few notations pertaining to symmetric products of manifolds, that we will use throughout the text.

\subsection{Symmetric products}\label{sectnot}

Let $X$ be a complex projective manifold. 
\begin{enumerate}\itemsep=0.5em
		\item For any $m \in \mathbb N^\ast$, we will denote by $X_m = \quotient{\mathfrak{S}_m}{X^m}$ the $m$-th symmetric product of $X$. We let $q : X^m \to X_m$ be the natural projection. Elements of $X_m$ will be denoted by $[x_1, x_2, ..., x_m]$ (where $(x_1, ..., x_m) \in X^m)$. Also, if $s>0,m_1,\dots ,m_s$ are positive integers such that $\sum_im_i=m$, and $x_1,\dots,x_s\in X$ are pairwise distinct, we write $[x_1^{m_1},\dots,x_s^{m_s}]:=[x_1,\dots,x_1,\dots,x_s,\dots,x_s]$, where each $x_i$ is repeated $m_i$ times, for $i=1,\dots,s$.

		\item	For any $V \subset X$ and any $1 \leq i \leq m$, we let $\mathfrak{d}_i (V) = \{ [x_1, ..., x_m] \in X_m | x_1, ..., x_i \in V \}$.

		\item For any $1 \leq i \leq m$, we let $\mathfrak{D}_i(X_m) = \{ [x_1, ..., x_m] \in X_m | x_1 = ... = x_i \}$ be the $i$-th diagonal locus. Note that ${\rm codim}\, \mathfrak{D}_i (X_m) = n(i - 1)$.
			
		\item For any divisor $A$ on $X$, we will denote by $A^\sharp = \sum\limits_{i=1}^m {\rm pr}_i^\ast A$ the associated $\mathfrak{S}_m$-invariant divisor on $X^m$, and by $A_\flat = q_\ast A^\sharp$ the projection on $X_m$. Note that $A_\flat$ is a Cartier divisor on $X_m$, hence it induces a well-defined line bundle.
\end{enumerate}

Remark that the construction $X \rightsquigarrow X_m$ is functorial, any holomorphic map $f~:~X~\to~Y$ inducing a natural holomorphic map $f_m : X_m \longrightarrow Y_m$.

\subsection{The Reid-Tai-Weissauer criterion} \label{sectcrit} For later reference, we now recall an important criterion for the extension of differential forms on resolutions of quotient singularities.
\medskip

Let $G$ be a group acting on a complex manifold $X$ of dimension $n$. The criterion can be stated in terms of the following condition: 
\medskip

\begin{bfseries} Condition $(\mathrm{I}_{x,d})$. \end{bfseries} Let $x \in X$, and let $d \in \mathbb N$. We say that the condition $(I_{x, d})$ is satisfied, if for any $g \in G - \{ 1 \}$ stabilizing $x$, the following holds. Assume that $g$ has order $r$, and let $(z_1, ..., z_n)$ be coordinates centered at $x$ such that $g$ acts by
$$
g \cdot (z_1, ..., z_n) = (\zeta^{a_1} z_1, ..., \zeta^{a_n} z_n),
$$
where $\zeta = e^{\frac{2i\pi}{r}}$, and $a_1, ..., a_n \in \llbracket 0, r - 1 \rrbracket$. Then, for any choice of $d$ distinct elements ${i_1}, ..., {i_d}$ in $\llbracket 1, n \rrbracket$, we have
$$
a_{i_1} + ... + a_{i_d} \geq r.
$$

It is useful to state a weaker condition under which the differentials will extend meromorphically to a resolution of singularities. Resume the same notations as before, and let $\alpha > 0$.
\medskip

\begin{bfseries} Condition $(\mathrm{I}'_{x,d, \alpha})$. \end{bfseries} We say that the condition $({\rm I}'_{x, d, \alpha})$ is satisfied, if the same statement as in Condition $({\rm I}_{x, d})$ holds, with the inequality replaced by  
$$
a_{i_1} + ... + a_{i_d} \geq r ( 1 - \alpha).
$$

\begin{proposition}[{\cite[Lemma 4. p. 213]{wei86}}] \label{propextension}
	Let $d \in \mathbb N$. Assume that the condition $({\rm I}_{x, d})$ (resp. $({\rm I}'_{x, d, \alpha})$) holds for any point $x \in X$. Let $Y = \quotient{G}{X}$, and let $\widetilde{Y}$ be a smooth resolution of singularities of $Y$. Let $Y^\circ$ be the smooth locus of $Y$.
	
	Then, for any $p \geq d$, and for any $q \in \mathbb N$, the sections of $(\bigwedge^p \Omega_{Y^\circ})^{\otimes q}$ extend to the whole $\widetilde{Y}$ (resp. extends as meromorphic section of $(\bigwedge^p \Omega_{\widetilde{Y}})^{\otimes q}$ with a pole of order at most $\lfloor q \alpha \rfloor$).
\end{proposition}

\begin{remark}
	1. The fact that $q$ is arbitrary in the criterion above is crucial. Note that if $q = 1$, then for any $p \geq 1$, any section of $\bigwedge^p \Omega_{Y^\circ}$ extends to $\widetilde{Y}$, e.g. by \cite{Freitag} or \cite{GKKP}. The proof of \cite{Freitag} consists essentially in remarking that $({\rm I'}_{x, d, \alpha})$ always holds for some $\alpha < 1$, so $\lfloor q \alpha \rfloor = 0$ in this case. 

	2. Proposition \ref{propextension} is a generalization of well-known criterion proved independently by Tai \cite{tai82} and Reid \cite{reid79} (which is simply the case $p = \dim X$). The proof given in \cite{wei86} is stated in the case where $X = \mathbb H_g$ is the Siegel upper half-space acted upon by $G = {\rm Sp}(2g, \mathbb Z)$, but it generalizes immediately to the general case (for more details in English, the reader can see e.g. \cite{cad18}).
\end{remark}
\medskip

\subsection{Jet differentials}\label{jets}

We will now recall some basic facts around the notion of jet differentials. For more details, the reader can refer to \cite{dem12a}.

Let $X$ be a complex manifold, and $k, m \in \mathbb N$ be integers. We will denote the unit disk by $\Delta$. The {\em Green-Griffiths vector bundle of jet differentials of order $k$ and degree $m$}, is the vector bundle  $E_{k, m}^{GG} \Omega_X \to X$, whose sections over a chart $U \subset X$ identify with differential equations acting on holomorphic maps $f : \Delta \to U$, with adequate order and degree. Writing $f = (f_1, ..., f_n)$ in local coordinates, $P(f)$ can be put under the form: 
$$
\begin{aligned}
	P(f) : t \in & \Delta \longmapsto P(f)(t) \\
	& = \sum_{I = (i_{1, 1}, ..., i_{1, k}, ..., i_{n, 1}, ..., i_{n,k})} a_I(f(t)) \, (f_1')^{i_{1, 1}} ... (f_n')^{i_{n, 1}} ...  (f_1^{(k)})^{i_{1, k}} ... (f_n^{(k)})^{i_{n, k}},
\end{aligned}
$$
and $P$ being of degree $m$ means that if $g(t) = f(\lambda t)$, then $P(g)(t) = \lambda^m P(f)(\lambda t)$.
\medskip

For any order $k \geq 1$, we can form the Green-Griffiths jet differential algebra $E_{k, \bullet}^{GG} \Omega_X = \bigoplus_{m \geq 0} E_{k, m} \Omega_X$, and define the {\em $k$-th jet space} $X_{k}^{GG} = {\bf Proj}_X(E_{k, \bullet}^{GG} \Omega_X)$. We check that the elements of $X_k^{GG}$ are naturally identified with {\em classes of $k$-jets}, i.e. $k$-th order Taylor expansion of holomorphic maps $f: (\Delta, 0) \to X$, up to linear reparametrization. Each jet space is endowed with a projection map $\pi_k : X_k^{GG} \to X$ and tautological sheaves $\mathcal O_{X_k^{GG}}(m)$ ($m \geq 0$), such that
$$
(\pi_k)_\ast \mathcal O_{X_k^{GG}}(m) = E_{k, m}^{GG} \Omega_X
$$
for any $m \geq 1$.

If $C$ is a complex curve, any map $f : C \to X$ admit well-defined lifts $f_{[k]} : C \to X_k^{GG}$ obtained by taking the $k$-th Taylor expansion at each point of $C$. The main interest of jet differential equations in the study of complex hyperbolicity comes from the following fundamental vanishing theorem, which permits to give strong restrictions on the geometry of entire curves.

\begin{theorem}[\cite{SY96, dem97}]
	Let $X$ be a complex projective manifold, and let $A$ be an ample line bundle on $X$. Let $k, m \geq 1$, and let $P \in H^0(X, E_{k, m}^{GG} \Omega \otimes \mathcal O(-A))$. Let $f : \mathbb C \longrightarrow X$. Then $f$ is a solution of the holomorphic differential equation $P$, i.e. $P(f; f', ..., f^{(k)}) = 0$.
\medskip

	In other words, for any entire curve $f : \mathbb C \to X$, we have $f_{[k]}(\mathbb C) \subset \mathbb{B}_+(\mathcal O_{X_k^{GG}}(1))$, where $\mathbb B_+$ denotes the augmented base locus.
\end{theorem}

The previous theorem has strong implications in cases where global jet differential equations are numerous. In these notes, we will be able to produce such differential equations using a basic variant of the {\em orbifold jet differentials} which were introduced by the second and third authors in a joint work with L. Darondeau \cite{CDR18}. We will explain briefly how these objects can be defined in our context at the beginning of Section \ref{sectjetdifforb}.

%% file: special.tex
\section{Special varieties}\label{special}
We collect here basic definitions and constructions related to special varieties, while referring to \cite{Cam04} for more details.

\subsection{Special Manifolds via Bogomolov sheaves}
Let \(X\) be a connected complex nonsingular projective variety of complex dimension \(n\). 
For a  rank-one coherent subsheaf \(\calL\subset\Omega^p_X\), denote by \(H^0(X,\calL^{m})\) the space of sections of \(\Sym^m(\Omega^p_X)\) which take values in \(\calL^{m}\) (where as usual $\calL^m \coloneqq \calL^{\otimes m}$).
The \emph{Iitaka dimension} of $\calL$ is \(\kappa(X,\calL)\bydef\max_{m>0}\{\dim(\Phi_{\calL^m}(X))\}\), i.e. the maximum dimension of the image of rational maps
\(\Phi_{\calL^{m}}\colon X\dasharrow\PP(H^0(X,\calL^{m}))\) defined at the generic point of \(X\), 
where by convention \(\dim(\Phi_{\calL^{m}}(X))\bydef-\infty\) if there are no global sections.
Thus \(\kappa(X,\calL)\in \left\{-\infty, 0,1,\dotsc, \dim(X)\right\}\). 
In this setting, a theorem of Bogomolov in \cite{Bog} shows that, if \(\calL\subset\Omega_{X}^{p}\), then \(\kappa(X,\calL)\leq p\).
\begin{definition}
  Let $p >0$. A rank one saturated coherent sheaf \(\calL\subset\Omega_{X}^{p}\)
  is called a \emph{Bogomolov sheaf} if \(\kappa(X,\calL)=p\), i.e. if $\calL$ has the largest possible Iitaka dimension.
\end{definition}

The following remark shows that the presence of Bogomolov sheaves on $X$ is related to the existence of fibrations $f: X\to Y$ where $Y$ is of general type.

\begin{remark}
  If \(f\colon X\to Y\) is a fibration (by which we mean a surjective morphism with connected fibers) and \(Y\) is a variety of general type of dimension \(p>0\), then the saturation of \(f^*(K_Y)\) in \(\Omega^p_X\) is a Bogomolov sheaf of \(X\), 
\end{remark}

The second author introduced the notion of specialness in \cite[Definition 2.1]{Cam04} to generalize the absence of fibration as above.

\begin{definition}\label{def:special}
  A nonsingular variety \(X\) is said to be \textsl{special} (or \textsl{of special type}) if there is no Bogomolov sheaf on $X$. 
  A projective variety is said to be \textsl{special} if some (or any) of its resolutions are special.
\end{definition}

By the previous remark if there is a fibration $X \to Y$ with $Y$ of general type then $X$ is nonspecial. In particular, if \(X\) is of general type of positive dimension, $X$ is not of special type.

\subsection{Special Manifolds via orbifold bases} 
It is possible to give a characterization of special varieties using the theory of \emph{orbifolds}. We briefly recall the construction.

Let $Z$ be a normal connected compact complex variety.
An \emph{orbifold divisor} \(\Delta\) is a linear combination \(\Delta\coloneqq\sum_{\left\{D\subset Z\right\}}c_{\Delta}(D)\cdot D\), 
where \(D\) ranges over all prime divisors of \(Z\),
the \emph{orbifold coefficients} are rational numbers \(c_{\Delta}(D)\in[0,1]\cap\Q\) such that all but finitely many are zero.
Equivalently,
\[
  \Delta
  =
  \sum_{\{D\subset Z\}}
  \left(
    1-\frac{1}{m_{\Delta}(D)}
  \right)
  \cdot
  D,
\]
where only finitely \textsl{orbifold multiplicities} \(m_{\Delta}(D)\in\Q_{\geq 1}\cup \{+\infty\}\) are larger than \(1\).

An orbifold pair is a pair \((Z,\Delta)\) where $\Delta$ is an orbifold divisor; they interpolate between the compact case where \(\Delta=\varnothing\) and the pair \((Z,\varnothing)=Z\) has no orbifold structure, and the {\it open}, or {\it purely-logarithmic case} where $c_j = 1$ for all $j$, and we identify \((Z,\Delta)$ with $Z\setminus\Supp(\Delta)\).

When \(Z\) is smooth and the support \(\Supp(\Delta)\bydef\cup D_j\) of \(\Delta\) has normal crossings singularities, we say that \((Z,\Delta)\) is {\it smooth}. When all multiplicities \(m_j\) are integral or \(+\infty\), we say that the orbifold pair \((Z,\Delta)\) is {\it integral}, and when every $m_j$ is finite it may be thought of as a virtual ramified cover of \(Z\) ramifying at order \(m_j\) over each of the \(D_j\)'s.\medskip

Consider a fibration \(f\colon X\to Z\) between normal connected complex projective varieties. In general, the geometric invariants (such as \(\pi_1(X), \kappa(X),\dots\)) of \(X\) do not coincide with the `sum' of those of the base (\(Z\)) and of the generic fiber (\(X_\eta\)) of \(f\).
Replacing \(Z\) by the `orbifold base' \((Z,\Delta_f)\) of \(f\), which encodes the multiple fibers of \(f\), leads in some favorable cases to such an additivity (on suitable birational models at least).

\begin{definition}[Orbifold base of a fibration]
  \label{dob} 
  Let \(f\colon (X,\Delta)\to Z\) be a fibration $X \to Z$ as above and let $\Delta$ be an orbifold divisor on \(X\). 
  We shall define the \textsl{orbifold base} \((Z,\Delta_{f})\) of \((f,\Delta)\) as follows: to each irreducible Weil divisor \(D\subset Z\) we assign the multiplicity \(m_{(f,\Delta)}(D)\bydef\inf_k\{t_k\cdot m_\Delta(F_k)\}\), where the scheme theoretic fiber of $D$ is \(f^*(D)=\sum_k t_k.F_k+R\), \(R\) is an \(f\)-exceptional divisor of \(X\) with \(f(R)\subsetneq D\) and \(F_k\) are the irreducible divisors of \(X\) which map surjectively to \(D\) via $f$.
\end{definition}
\begin{remark}
  Note that the integers \(t_k\) are well-defined, even if \(X\) is only assumed to be normal.
\end{remark}

Let \((Z,\Delta)\) be an orbifold pair. Assume that \(K_Z+\Delta\) is \(\Q\)-Cartier (this is the case if \((Z,\Delta)\) is smooth, for example): we will call it the \emph{canonical bundle} of \((Z,\Delta)\). Similarly we will denote by the {\it canonical dimension} of $(Z,\Delta)$ the Kodaira dimension of $K_Z + \Delta$ i.e. $\kappa(Z,\Delta): = \kappa(Z,\calO_Z(K_Z+\Delta))$. Finally, we say that the orbifold \((Z,\Delta)\) is of {\it general type} if \(\kappa(Z,\Delta)=\dim(Z)\).

\begin{definition}
A fibration \(f\colon X\to Z\) is said to be \textsl{of general type} if \((Z,\Delta_{f})\) of general type.
\end{definition}

The non-existence of fibrations of general type in the above sense turns out to be equivalent to the specialness condition of Definition \ref{def:special}.

\begin{theorem}[\protect{see \cite[Theorem 2.27]{Cam04}}]
  A variety \(X\) is special if and only if it has no fibrations of general type.
\end{theorem}

Let us now recall the existence of the \emph{core} map  (see \cite[Section 3]{Cam04} for details). Given a smooth projective variety $X$ there is a functorial fibration $c_X: X \to C(X)$, called the \emph{core} of $X$ such that the fibers of $c_X$ are special varieties and the base $C(X)$ is either a point (if and only if $X$ is special) or an orbifold of general type.

As mentioned in the introduction, the second author has proposed the following generalizations of Lang's conjectures.

\begin{conjecture}[Campana]\label{specialconj}
\begin{enumerate}
\item Let $X$ be a complex projective variety. Then, $X$ is special if and only if there exists an entire curve $\C \to X$ with Zariski dense image.
\item Let $X$ be a projective variety defined over a number field.
Then, the set of rational points on $X$ is potentially dense if and only if $X$ is special.
\end{enumerate}
\end{conjecture}

Finally, let us remark that previous conjectures (see \cite[Conjecture 1.2]{HaT}) proposed to characterize potential density with the weaker notion of \emph{weak specialness}.

\begin{definition}
  A projective variety \(X\) is said to be \textsl{weakly special} if there are no finite étale covers \(u\colon X'\to X\) admitting a dominant rational map \(f'\colon X'\to Z'\) to a positive dimensional variety $Z'$ of general type. 
\end{definition} 

It has been shown in \cite{CP07} and \cite{RTW} that one cannot replace ``special'' by ``weakly-special'' in Conjecture \ref{specialconj} in the analytic and function fields settings.

%% file: canonical.tex
\section{Canonical fibrations}

We will now study conditions under which various canonical fibrations are preserved by the symmetric product. In the rest of the text, a \emph{fibration} will be a surjective morphism with connected fibres. Then, if $f : X \to Y$ is a fibration, so is $f_m : X_m \to Y_m$. 
\medskip

We shall consider the following (bimeromorphically well-defined) fibrations for $X$ smooth compact of dimension $n$:

\begin{enumerate}
	\item The Moishezon-Iitaka fibration $f:=J:X\to B$

Assuming $X$ to be smooth K\"ahler:

\item The `rational quotient'\footnote{Also termed MRC fibration.} $f:=r:X\to B$.

\item The `core map' $f:=c:X\to B$. 
\end{enumerate}

Recall that \cite{AA02} shows that if $X$ is smooth, and if $\dim X \geq 2$, the singularities of $X_m$ are canonical, and consequently, that $\kappa(X_m)=m. \kappa(X)$.

The goal is to extend (and exploit) \cite{AA02} in order to show the following:

\begin{theorem} Assume $\dim B\geq 2$, then in each of these $3$ cases ($f=J,r,c$ respectively),  $f_m:X_m\to B_m$ is the same fibration ($J_m,r_m,c_m$ respectively), with $X_m,B_m$ replacing $X,B$. (In the case of the rational quotient, or of the core map, there are exceptional cases when $B$ is a curve. See Theorems \ref{rq} and \ref{c_m} below, as well as Remark \ref{n=1}).\end{theorem}

	\begin{remark} \label{n=1} The conclusion is obviously false when $\dim X=1$ and $g(X) \geq 2$, since $q_m:X^m\to X_m$ then ramifies in codimension $n=1$. One recovers a uniform statement by equipping $X_m$ with its natural orbifold structure, obtained by assigning to each component $D_{j,k}$ in $X_m$ of the diagonal locus $\mathfrak{D}_2(X_m)$ its natural multiplicity $2$. The orbifold divisor $D_m:=\sum_{j<k}(1-\frac{1}{2}).D_{j,k}$ on $X_m$ has then the property that $q_m^*(K_{X_m}+D_m)=K_{X^m}$. In particular, $\kappa(X_m,K_{X_m}+D_m)=m.\kappa(X)$. The divisor $D_m$ will appear again when we consider the core map below. Notice however that, as soon as $m\geq 3$, the orbifold divisor $D_m$ is not of normal crossings (for $m=3$ for example, it is locally analytically a product of of disk by a plane cusp.) \end{remark}

For $f=J$, the proof is an immediate consequence of \cite{AA02}. Indeed: the general fibre of $f_m$ is a product of fibres of $J$, hence has $\kappa=0$. On the other hand, $\kappa(X_m)=m.\kappa(X)=\dim(B_m)$. The conclusion follows. 

\medskip

Before starting the study of $c_m,r_m$, let us make some simple observations on $f_m:X_m\to B_m$ if $f:X\to B$ is a fibration (with connected fibres) between two connected compact complex manifolds.

1. The generic fibre of $f_m$ over a point $[b_1,\dots,b_m]\in B_m$ is isomorphic to the (unordered) product $X_{b_1}\times\dots X_{b_m}$ if the $b_i$ are pairwise distinct. In particular, if the generic fibre of $f$ is rationally connected, or special, so are the generic fibres of $f_m$.

2. if the schematic fibres $X_{b_i}$ are reduced, so is the fibre over $[b_1,\dots,b_m]$, whatever the $b_i$.

3. If $f$ has a local section over a neighborhood of each of the $b_i's$, $f_m$ has (an obvious) local section over a neighborhood of $[b_1,\dots,b_m]$.

\medskip

We shall now prove the statement for the other two fibrations.

\subsection{The `rational quotient'}

\begin{theorem}\label{rq} Let $r:X\to B$ be the rational quotient map of $X$, compact K\"ahler. Then $r_m:X_m\to B_m$ is the rational quotient map of $X_m$ if $\dim(B)\neq 1$.
If $B$ is a curve of genus $g>0$, and $R_m:X_m\to R(m)$ is the rational quotient map, there are two cases: either $m<g$, then $R_m=r_m, R(m)=B_m$, or $R_m=jac^m_B\circ r_m:X_m\to \Jac(B)$, where $jac^m_B:B_m\to \Jac(B)$ is the natural Jacobian map.
 \end{theorem}

\begin{proof} Recall that $r$ is characterised by the fact that its fibres are rationally connected and (a smooth model of) its base is not uniruled (by \cite{GHS}). Since the generic fibres of $r_m$ are products of fibres of $r$, hence rationally connected, it is sufficient to show that a smooth model $\mu:B'_m\to B_m$ of $B_m$ is not uniruled if $B$ is {\bf not} a curve of positive genus, case treated now. Assume it were, we would then have an algebraic family of generically irreducible curves $C'_t$ covering $B'_m$ and with $-K_{B'_m}.C'_t>0$. Since the singularities of $B_m$ are canonical, this implies $K_{B_m}.C_t<0$, where $C_t:=\mu_*(C_t)$, since $K_{B'_m}=\mu^*(K_{B_m})+E'$, with $E'$ effective, by \cite{AA02}. The conclusion now follows, using \cite{Miyaoka-Mori}, from the fact that $K_{B^m}=(q_m^B)^*(K_{B_m})$ is pseudo-effective (ie: has nonnegative intersection with any covering algebraic family of generically irreducible curves), by lifting to $B^m$ the generic curve $C_t$.

Assume now that $B$ is a curve of genus $g>0$. Then $jac^m_B:B_m\to \Jac(B)$ has connected fibres generically projective spaces of dimension $0$ if $m\leq g$, and positive dimension if $m>g$. Moreover the image of $jac^m_B$ is never uniruled when $m>0$. This shows the claim, by \cite{GHS}.
\end{proof}

\begin{corollary} $X$ is rationally connected if and only if so is $X_m$ for some $m$. $X$ is uniruled if and only if so is $X_m$ for some $m$, unless we are in the following situation, where $X_m$ is uniruled, but $X$ is not: $X$ is a curve of genus $g>0$, and $m>g$.  \end{corollary}

\begin{proof} Indeed: the uniruledness (resp. rational connectedness) of $X$ is characterised by: $\dim(X)>\dim(B)$ (resp. $\dim(B)=0)$, and $\dim(B_m)=m.\dim(B)$. We thus see that $X_m$ is rationally connected (resp. uniruled) if so is $X$. Conversely, the preceding Theorem \ref{rq} shows that the claim holds true if $\dim(R(m))=\dim(B_m)=m.\dim(X)$. This is the case unless possibly when $r:X\to B$ fibres over a curve $B$ with $g(B)>0$, and $m> g$. In this case, $X_m$ is uniruled, but not rationally connected. Thus $X_m$ rationally connected implies $X$ rationally connected. On the other hand, if $X$ is not  uniruled, we have $X=B$ is a curve, and so $X_m$ is uniruled if and only if $m>g$. Hence the corollary.\end{proof}

\begin{remark} If $X$ is unirational, so is obviously $X_m$, for any $m>1$. It is true, but less obvious (\cite{mat68}), that if $X$ is rational, then so is $X_m$, for any $m>1$. From this follows that if $X$ is stably rational, then so is $X_m$, for $m>1$ too. This naturally leads to consider the converses.
\end{remark}

\begin{question} Assume that $X_m$ is unirational (resp. rational, stably rational) for some $m\geq 2$, is then, yes or no, $X$ unirational (resp. rational, stably rational)? If some $X_m,m>1$ is rational, is $X$ unirational? Some specific cases are as follows. \end{question}

\begin{example} 1. If $X$ is a smooth cubic hypersurface of dimension $n\geq 3$, is $X_m$ rational for some large $m$?

2. If $X$ is the double cover of $\Bbb P^3$ ramified over a smooth sextic surface, $X$ is Fano, hence rationally connected, but its unirationality (or not) is an open problem. Is $X_m$ unirational for some large $m$? The same question arises for $X$ a conic bundle over $\Bbb P^2$ with a smooth discriminant of large degree.
\end{example}

\subsection{The core map}

\begin{theorem}\label{c_m} If $c:X\to B$ is the core map of $X$, then $c_m:X_m\to B_m$ is (bimeromorphically) the core map of $X_m$ if $n\geq 2, p\neq 1$, where $p:=\dim(B)$ is the dimension of the core. \end{theorem} 

The case where $B$ is a curve is studied in the next subsection.

\begin{corollary} If $n\geq 2$, $X$ is special if and only if so is $X_m$ for some $m$ (unless when $n\geq 2$, and the core of $X$ is a curve. See Remark \ref{n=1}).\end{corollary}

Indeed, $X$ (resp. $X_m)$ is special if and only if $\dim(B)=0$ (resp. $\dim(B_m)=0)$, and $\dim(B_m)=m.\dim(B)$.

\begin{proof} [Proof of Theorem \ref{c_m}] Since the general fibres of $c_m$ are products of special manifolds they are special (it is easy to see that  a product of special manifolds is special). It is thus sufficient to show that the `neat orbifold base' of $c_m$ is of general type, knowing that so is the neat orbifold base of $c$. This requires some explanation.

Recall that $f:X\to B$ is neat if there exists a bimeromorphic map $u:X\to X_0, X_0$ smooth, such that each $f$-exceptional divisor is also $u$-exceptional, and the complement of the open set $U=B\setminus D\subset B$ over which $f$ is submersive is a snc divisor, as well as $f^{-1}(D)\subset X$. Such a neat model of $f_0:X_0\dasharrow B$ is obtained by flattening $f_0$, followed by suitable blow-ups. In this case, the support of $D_f$, the orbifold base of $f$, is snc too, and $\kappa(B,D_f)$ is minimal among all bimeromorphic models of $f$. More precisely, $\kappa(B,D_f)=\kappa(X,L_f)$, where $L_f:=f^*(K_B)^{sat}\subset \Omega^p_X$, where $p:=\dim(B)$, and $f^*(K_B)^{sat}$ is the saturation of $f^*(K_B)$ in $\Omega^p_X$. See \cite{Cam04} for details. Notice also that if $c:X\to B$ is a neat model of some $f_0:X_0\dasharrow B_0$, and if $x\in X$ is any point, there is another neat model $f':X'\to B'$ dominating\footnote{In the sense that there exists birational maps $u':X'\to u$ and $\beta':B'\to B$ such that $f\circ u'=\beta'\circ f'$.}$f:X\to B$ such that $x$ does not belong to any $f'$-exceptional divisor on $X'$, and lies in the image of the smooth locus of the reduction of a fibre of $f'$. If this condition is not realised on $(X,f)$ it is then sufficient to suitably blow-up $X$, then flatten the resulting map by modifying $B$, and finally take a smooth model of the resulting $f$. The claim of Theorem \ref{c_m} then holds true for $(X,f)$ if it holds for $(X',f')$. 

Let $c:X\to B$ be neat with respect to $u:X\to X_0$, and let $c_m:X_m\to B_m$, together with a smooth model $c'_m:X'_m\to B'_m$ of $c_m$ (ie: $X'_m, B'_m$ are smooth models of $X_m,B_m)$.

Let us prove first that $c^m:X^m\to B^m$ is the core map of $X^m$, with orbifold base $(B^m,D_{f^m})$ and Kodaira dimension $m.\kappa(B,D_f)$. This follows inductively on $m$ from the following easy lemma, which also shows that $D_{f^m}=\cup_{s\in S_m}s(D_f\times X^{m-1})$.

\begin{lemma} Let $f:X\to V, g:Y\to W$ be neat fibrations with orbifold bases $(V,D_f), (W,D_g)$. Then $f\times g: X\times Y\to V\times W$ is neat, its orbifold base is $(X\times Y, D_f\times W+V\times D_g)$, and its Kodaira dimension is $\kappa(V,D_f)+\kappa(W,D_g)$. \end{lemma}
\begin{proof} If $E\subset V\times W$ is an irreducible divisor mapped surjectively on both $V$ and $W$, there is only one irreducible divisor $F\subset X\times Y$ such that $(f\times g)(F)=E$, which has multiplicity $1$ in $(f\times g)^*(E)$, since over $(v,w)\in E$ generic, $(f\times g)^{-1}(v,w)=X_v\times Y_w$, reduced. The other conclusions are obtained by a similar argument.
\end{proof}

$\bullet$ We now turn to the proof of Theorem \ref{c_m}. Let $c_m:X_m\to B_m$ be deduced by quotient from the core map $c^m$, and let $D_{c_m}\subset X_m$ be the direct image of $D_{c^m}$ under the quotient map $q_B:B^m\to B_m$, so that $D_{c^m}=(q_B)^*(D_{c_m})$. 
  It is sufficient to show that $\rho^*(c_m^*((K_{X_m}+D_{c_m})^{\otimes k}))\subset Sym^k(\Omega^{m.p}_{X'_m})$ for any (or some) $k>0$ such that $k.(K_{X_m}+D_{c_m})$ is Cartier, where $\rho:X'_m\to X_m$ is a smooth model of $X_m$. 

$\bullet$ If $p:=\dim(B)=0$, there is nothing to prove.

$\bullet$ We thus assume that $p:=\dim(B)\geq 2$. The problem is local (in the analytic topology) on $X^m,X_m,B^m,B_m$. By the observations made above, we shall assume that the points $(x_1,\dots,x_m)$ near which we treat the problem do not belong to any $c$-exceptional divisor, and are regular points of the reduction of the fibre of $c$ containing them. The fibration $c$ is thus given in suitable local coordinates on $X$ and $B$ by the map $c:(x_1,\dots, x_n)\to (b_1,\dots b_p)$ with $b_i:=x_i^{t_i},\forall i=1,\dots, p, p<n$, where the support of $D_c$ is contained in the union of the coordinate hyperplanes $b_i=0$ of $B$, the multiplicity of $b_i=0$ in $D_c$ being an integer $t'_i$, with $1\leq t'_i\leq t_i,\forall i\leq p$, by the very definition of the orbifold base.

	Since $c^*\big(\Big(\frac{db_i}{b_i^{1-(1/t'_i)}}\Big)^{\otimes t'_i}\Big)=t_i^{t'_i}.x_i^{(t_i-t'_i)}.(dx_i)^{\otimes t'_i}$, we see that $(K_B+D_c)^{\otimes t}$ is Cartier and $c^*((K_B+D_c)^{\otimes t})\subset Sym^t(\Omega^p_X)$, if $t={\rm lcm}\{t'_i\}$. 

Thus $(c^m)^*((K_{B^m}+D_{c^m})^{\otimes t})\subset Sym^t(\Omega^{pm}_{X^m})$, this natural injection being deduced from the description of $D_{c^m}$ given above (which shows that it is snc since so is $D_c)$. The saturation of the image of this injection inside $Sym^t(\Omega^{pm}_{X^m})$ is the line bundle generated by $T:=(w_1\wedge \dots\wedge w_m)^{\otimes t}$, where $w_j:=dx_{1,j}\wedge\dots\wedge dx_{p,j}, \forall j=1,\dots,m$. Here $(x_{1,j},\dots,x_{n,j})$ are the local coordinates near the point $z_j\in X$, on the $j$-th component $X_j\cong X$ of $X^m$ near the point $(z_1,\dots,z_m)$.

It is sufficient (considering separately the distinct points of the set $\{z_1,\dots,z_m\})$ to deal with the case where $z_j=z_k, \forall j,k\leq m$.

	The operation of $\mathfrak{S}_m$ on the coordinates $x_{i,j}, i\leq n, j\leq m$ fixes the set of coordinates $x_{i,j}, i\leq p, j\leq m$ and induces on the vector space $\oplus_jV_j:=\oplus_{i,j}\Bbb C.x_{i,j}, j\leq p$ they generate a representation which is a direct sum of $p$ copies of the regular representation. 

	The conclusion then follows from Proposition \ref{propextension}. One checks the conditions\footnote{The simpler form of  our tensor $T$ reduces the conditions, for a given $g$, in the proof-not the statement-of Lemma 4 of \cite{wei86} to a single one: $\sigma(g)\geq r$ (in loc.cit the data $\ell, d, N, m$ correspond to $t, pm, n, r$ here, respectively.)} given in \cite{wei86} by using the (purely algebraic) proof of Prop.1, p. 1370, of \cite{AA02}, which says that if $\rho:{\mathfrak{S}}_m\to Gl(\oplus_{j=1}^{j=m}V)$ is a representation which is the direct sum of $p$ copies of the regular representation, where $V$ is a complex vector space of dimension $p\geq 2$, then $\sigma(g)=\frac{n}{2}.r.(\sum_{k=1}^{k=s}(r_k-1))\geq r$, for any $g\in \mathfrak{S}_m$ which is the product of $s$ non-trivial disjoint cycles of lengths $r_k$, and $r:={\rm lcm} (r_k's)$ is the order of $g$.
Here $\sigma(g):=\sum_h a_h$, if the eigenvalues of $\rho(g)$ are $\zeta^{a_h}$, where $\zeta$ is any complex primitive $r$-th root of the unity, and $0\leq a_h<r$ for any $h$. \end{proof}

\subsection{The core map of $X_m$ when the base of $c$ is a curve.}

We now assume that $p:=\dim(B)=1$. Let $c:X\to B$ be the core map, and $(B,D_c)$ its orbifold base. When $D_c=0$, the situation is easy:

\begin{theorem} \label{thmcorecurve} Assume that the core map $c:X\to B$ maps onto a curve $B$, and that its orbifold-base divisor $D_c=0$. Then $c_m:X_m\to B_m$ is the core map if $m<g$, and $X_m$ is special if $m\geq g$.
\end{theorem}

\begin{proof} Since $D_c=0$, the fibration $c:X\to B$, and so $c_m$, has everywhere local sections, thus the same is true for $c_m$, and hence for any smooth birational model of $c_m$. The conclusion thus follows from the fact that $B_m$ is of general type if $m<g$, and special if $m\geq g$.\end{proof}

In the general case, we have a weaker statement:

\begin{corollary} If $c:X\to B$ is the core map, with $B$ a curve, there is an integer $g(B,D_c)>0$ such that $X_m$ is special if $m\geq g(B,D_c)$\end{corollary}
\begin{proof}

By assumption, the orbifold curve $(B,D_c)$ is of general type, hence `good', meaning that there exists a finite Galois cover $h:\tilde{B}\to B$ which ramifies at order $t'$ over each point $b\in D_c\subset B$, $b$ of multiplicity $t'$ in $D_c$. The normalisation $H:\tilde{X}\to X$ of the fibre-product $X\times_B \tilde{B}$ comes equipped with $\tilde{c}:\tilde{X}\to \tilde{B}$, which is its core map, since this fibration has everywhere local sections. 

If $m\geq g(\tilde{C})$, then $\tilde{X}_m$, and so also $X_m$, is special. This shows the claim. \end{proof}

\begin{remark} It would be interesting to know a precise bound for $g(B,D_c)$, such that $X_m$ is special for $m\geq g(B,D_c)$, and such that $c_m:X_m\to B_m$ is the core map for $m<g(B,D_c)$.
\end{remark}


%% file: entirecurves.tex
\section{Dense entire curves in symmetric powers}

\subsection{Dense entire curves in $Sym^m(G\times C)$}

Let $G$ (resp. $C)$ be a curve of genus $g(G)\leq 1$ (resp. $g(C)>1)$, and  $S=G\times C$, then $S_m$ is special if and only if $m\geq g$, which we assume from now on.  
Theorem \ref{thmcorecurve} shows that $S_m$ is `special' (hence `weakly-special'), while of course, $S^m$ is not `weakly special'.
This section is devoted to the proof of a result stating that $S_m$ contains (lots of) entire curves $h:\C\to S_m$ with dense (not only Zariski-dense) image.
This statement was suggested by Ariyan Javanpeykar as a test case for the conjecture by the second named author, that special manifolds should contain dense entire curves. The arithmetic counterpart were that $S_m$ is `potentially dense' if defined over a number field.

\begin{theorem} If $S=G\times C$ is as above, and if $m\geq g$, then $S_m$ contains dense entire curves.
\end{theorem}

\begin{proof}We shall assume here that $G=\Bbb P_1$, the proof when $G$ is an elliptic curve being completely similar (just replacing $\Bbb C\subset \P_1$ by $\Bbb C\to G$ the universal cover). Observe that $C_m$ contains dense entire curves, since it fibres over ${\rm Jac}(C)$ as a $\Bbb P^r$-bundle, with $r:=m-g$, over the complement in ${\rm Jac}(C)$ of a Zariski-closed subset of codimension at least $2$.

	Take a dense entire curve $f:\Bbb C\to C_m$, let $V\subset \Bbb C\times C$ be the graph of the family of $m$-tuples of points of $C$ parameterized by $\Bbb C$ via $f$ (ie: $V:=\{w:=(z,c)\vert c\in C, c\in f(z)\}$. The map $\pi:V\to \Bbb C$ sending $w=(z,c)$ to $z$ is thus proper, open and of geometric generic degree $m$. In particular, $V$ is a Stein curve (not necessarily irreducible). Let $F:V\to C$ be the projection on the second factor. Let $g:V\to \Bbb C\subset \P_1=G$ be any holomorphic map. The product map $g\times F: V\to \Bbb C\times C\subset G\times C=S$ is thus well-defined. We now define the map $h:\Bbb C\to Sym^m(S)$ by sending $z\in \Bbb C$ to the $m$-tuple of $S$ defined by: $(g\times F)(\pi^{-1}(z))\subset S$. 
	
	We now just need to check that the map $g:V\to \Bbb C$ can be chosen such that $h(\Bbb C)\subset Sym^m(S)$ is dense there. Note first that if $(z_n)_{n>0}$ is a any discrete sequence of pairwise distinct complex numbers such that $\pi:V\to \Bbb C$ is unramified over each $z_n$, and if, for each $n>0$, $(t_{n,1},\dots, t_{n,m})$ is an $m$-tuple of complex numbers, there exists a holomorphic map $g:V\to \Bbb C$ such that $g(w_{n,i})=t_{n,i}, \forall n>0,i=1,\dots,m$, where $(w_{n,1}=(z_n,c_{n,1}),\dots, w_{n,m}=(z_n,c_{n,m}))=\pi^{-1}(z_n)$, and $(c_{n,1},\dots,c_{n,m}):=f(z_n)\in Sym^m(C)$ (the ordering being arbitrarily chosen).

	It is now an elementary topological fact that the sequences $(t_{n,1},\dots,t_{n,m}),n>0$ can be chosen in such a way that the sequence $(s_{n,1},\dots,s_{n,m})_{n>0}\in S^m$ is dense in $S^m$, where $s_{n,i}:=(t_{n,i},c_{n,i})\in S,\forall n>0,i=1,\dots,m$.

\end{proof}

\begin{remark} The preceding arguments work more generally for $X=G\times C$, when $C,m$ are as above, but $G$ enjoys the following property: for any smooth complex Stein curve $V\to \Bbb C$ proper over $\Bbb C$, and any sequence of distinct points $w_n\in W, t_n\in G$, there exists a holomorphic map $g:V\to G$ such that $g(w_n)=t_n,\forall n$.

	This property is satisfied for $G$ a complex torus or a unirational manifold. The same arguments would show the same result for $G$ rationally connected if one could answer positively the following question, answered positively in  \cite{CW}, when $V=\Bbb C$:

{\bf Question:} For $m,C,\pi:V\to \Bbb C$ defined as above, let $w_n\in V, t_n\in G,n\in \Bbb Z_{>0}$ be a sequence of points. Assume that the points $\pi(w_n)\in \Bbb C$ are all pairwise distinct. Does there exist a holomorphic map $g:V\to G$ such that $g(w_n)=t_n,\forall n$ if $G$ is rationally connected?

\end{remark}

\begin{remark}
Let now $\Delta^{(m)}\subset S^m$ be the `small diagonal', consisting of $m$-tuple of points of which $2$ at least coincide.
Thus $(S^m)^*:=S^m\setminus \Delta^{(m)}$ admits a surjective (but non-proper...) map to $C^m$.

	Let $\Delta_m\subset S_m$ be defined as: $\Delta_m:=  \mathfrak{D}_2(S_m) = q(\Delta^{(m)})$. We thus have, too: $\Delta^{(m)}=q^{-1}(\Delta_m)$.
	The restricted map $q : (S^m)^*\to (S_m)^*:=S_m\setminus \Delta_m$ is thus proper and \'etale.

Let $d_{(S^m)^*}:=d_{S^m\vert (S^m)^*}$ (by \cite{Kobayashi}) be the Kobayashi pseudometric on $(S^m)^*$. Since the Kobayashi pseudometric on
$S^m$ is the inverse image by $\pi$ of that on $C^m$, any entire curve $h:\C \to S^m$ (and so even more in $(S^m)^*$ has to be contained in some fibre of $\pi$.
Moreover, the Kobayashi pseudometric on $(S_m)^*$ is comparable to its inverse image in $(S^m)^*$ (and can be explicitly described).
This shows that any entire curve in $(S_m)^*$ is contained in the image by $q$ of a fibre of $\pi$, and is in particular algebraically degenerate
(although there are lots of dense entire curves on $S_m$, none of these avoids $\Delta_m)$.

This gives a counterexample to an analytic version of the `puncture problem' of \cite{HT}, similar to the arithmetic one of \cite{L}.

\end{remark}

\subsection{$\C^{2g}$-dominability of $S^{[g]}$, the $g$-th symmetric product of generic projective $K3$-surfaces}

 Let $S$ be a smooth projective $K3$-surface with\footnote{with some more work, it is probably possible to extend the next result to any projective $K3$-surface, by taking for $L$ an ample and primitive line bundle with $g$ minimal.} $Pic(S)\cong \Bbb Z$, generated by an ample line bundle $L$ of degree $2(g-1),g>1$. Such $K3$-surfaces are thus generic among projective $K3$-surfaces admitting a primitive ample line bundle of degree $2.(g-1)$. 
 
The objective is to prove the following

 \begin{theorem}\label{C^2-dom K3} For any such $S$, there is a (transcendental) meromorphic map $h:\C^{2g}\dasharrow Sym^g(S)$ whose image contains a nonempty Zariski open subset $U$ of $Sym^g(S)$ (such $2g$-fold is said to be `$\C^{2g}$-dominable'). In particular, for any countable subset $P$ of $U$, there is an entire curve on $Sym^{g}(S)$ whose image contains $P$. If $P$ is dense in $Sym^g(S)$, so is the image of this entire curve.
 \end{theorem}

 \begin{remark}1. The proof rests on a suitable abelian fibration $Sym^g(S)\dasharrow \P^g$.  Our result may thus be seen as analog to the case when $S$ is an elliptic $K3$ surface (over $\P_1)$ and $g=1$, shown in \cite{BL00}. 
 
 2. Our result is analogous to the arithmetic situation treated by \cite{HT}.
 
 3. Since $Sym^g(S)$ is special, \ref{C^2-dom K3} solves in a stronger form the conjecture of \cite{Cam04}.
  
 4. One may expect the conclusion of Theorem \ref{C^2-dom K3} to hold for $S^{[k]}$, any $k>1$ and any $K3$-surface (projective or not).   \end{remark}

 Before starting the proof, we recall some of the objects which have been attached to such a pair $(S,L)$.
 
\medskip

{\bf The $g$-th Symmetric product:} It is the direct product $S^g$ of $g$ copies of $S$, the symmetric product $Sym^g(S)$ is the quotient $q:S^g\to Sym^g(S)$ of $S^g$ by the symmetric group $S_g$ operating naturally on the factors. The Hilbert scheme $S^{[g]}$ of points of length $g$ on $S$, equipped with the Hilbert to Chow birational morphism $\delta:S^{[g]}\to Sym^g(S)$, is known to be smooth (\cite{F}, Theorem 2.4) and holomorphically symplectic \cite{beauville1}).

\medskip

{\bf The Relative Jacobian:} The line bundle $L$  determines $\P(H^0(S,L))^*):=\P^g$, the $g$-dimensional projective space (by Riemann-Roch and Kodaira vanishing). The linear system $\vert L\vert$ is base-point free and the associated map $\varphi:S\to \P^g$ is an embedding for $g\geq 3$ (a double cover ramified over a sextic for $g=2)$. For each $t\in \P^g$, the corresponding zero locus of a non-zero section of $\vert L\vert$ is an irreducible and reduced  (by the cyclicity  of $Pic(S)$ assumption) curve of genus $g$ denoted $C_t$. The incidence graph of this family of curves is denoted by $\gamma:\mathcal{C}\to \P^g$. For $d\in \Bbb Z$, the relative Jacobian fibration $j^d: J^d\to \P^g$ has fibre over $t$ the Jacobian $J^d_t$ of degree $d$ line bundles on $C_t$. The Jacobian $J^0_t$ of degree $0$ line bundles on $C_t$ (isomorphic to $J^d_t$ by tensorising with any given line bundle of degree $d)$  is a complex Hausdorff Lie group of dimension $g$ quotient of $H^1(C_t,\mathcal{O}_{C_t})$ by the (closed) discrete subgroup $H^1(C_t,\Bbb Z)$ (\cite{BPV}, II.2, Proposition (2.)). Thus, denoting with $j^0:J^0\to \P^g$ the relative Jacobian of degree $0$ (instead of $d)$ line bundles on the $C_t's$, and $V:=R^1\gamma_*(\mathcal{O}_{\mathcal{C}})\to \P^g$, this sheaf is locally free and thus a vector bundle $w:V\to \P^g$ of rank $g$ on $\P^g$. By \cite{Gro}, Theor\`eme 3.1, the relative Picard scheme is separated, and so the relative discrete group $R^1\gamma_*(\Bbb Z)\to \P^g$ is closed in $V$. Taking the quotient, we get:

\begin{lemma}\label{J^0} There is a holomorphic and surjective unramified map $H:V\to J^0$ over $\P^g$.\end{lemma}

{\bf The compactified Jacobian:} For $d\in \Bbb Z$, this is the compactification $\bar{j}^d:\overline{J^d}\to \P^g$ of $J^d$ over $\P^g$ obtained as a component of the moduli space of simple sheaves on $S$ (\cite{Mu}). This variety is compact smooth, holomorphically symplectic and, for $d=g$, birational to $S^{[g]}$ (\cite{beauville2}, Proposition 3). We denote with $\sigma: S^{[g]}\dasharrow \overline{J^g}$ this birational equivalence.

{\bf The covering by singular elliptic curves.} By \cite{BPV}, VIII, Theorem 23.1 (see references there for the original proofs), there is a nonempty curve in $\P^g$ parametrising (singular) curves $C_t's$ with elliptic normalisations. This family (and each of its components) covers $S$. Choosing $g$ generic (normalised) members $E_1,\dots,E_g$ of such an irreducible family provides a product $\varepsilon: E:=E_1\times \dots\times E_g\subset S^g$. By \cite{HT}, proof of Theorem 6.1, the composed projection $\bar{j}^g\circ \sigma\circ \varepsilon:E\to \P^g$ is a (meromorphic) multisection of the (meromorphic) fibration $\tau:=(\bar{j}^g)\circ \sigma: S^{[g]}\to \P^g$. This fact is in fact easy to prove, since if $C_t$ is smooth, it cuts each of the $E_i's$ in finitely many distinct points, and so the intersection of $E$ with $C_t^g$ is finite, and surjective on the fibre of $Sym^g(S)$ over $\P^g$.

\begin{proof} We can now prove Theorem \ref{C^2-dom K3}. For any complex manifolds $M,R$ equipped with a holomorphic map $\mu:M\to \P^g, r:R\to \P^g$, we denote with $R(M):=R\times_{\P^g}M$, equipped with the projections $\mu_M:R(M)\to M, r_M:R(M)\to R$ . This, applied to $R=V, R=\overline{J^d}, R=S^{[g]}(M)$, gives the fibre products $V(M),\overline{J^d}(M), S^{[g]}(M)$.

We have two meromorphic and generically finite maps $\epsilon: E\dasharrow S^{[g]}$, and $\sigma \circ \epsilon:E\dasharrow \overline{J^g}$. Denote with $E_t$ the fibre of $E$ over $t\in \P^g$. We get a birational map $\beta:\overline{J^g}(E)\dasharrow \overline{J^0}(E)$ over $E$ by sending a generic pair $(j,(e_1,\dots,e_g)_t)\in J_t^g\times E_t$ to $j\otimes \lambda^{-1}$, if $\lambda\in J_t^g$ is the line bundle on $C_t$ with a nonzero section vanishing on the $g$ points $e_i$.

Let $\pi:E'\to E$ be a modification making these maps holomorphic. Let $w_E:V(E')\to E'$ be the rank-$g$ vector bundle on $E'$ lifted from $w:V\to \P^g$. We get also a natural holomorphic map, unramified and surjective $H_E:V(E')\to J^0(E')$ over $E'$. Let $\mathcal E:=\pi_*(V)$:this is a rank-$g$ coherent sheaf on $E$, and there is a natural evaluation map: $\pi^*(\mathcal{E})\to V$ over $E'$.

Let now $\rho: \overline{E}\to E$ be the universal cover, so that $\overline{E}\cong \C^g$. Let $\pi':E'\times _E\overline{E}\to \overline{E}$ be deduced from $\pi:E'\to E$ by the base change $\rho$. Hence $\pi'$ is a proper modification. The sheaf $\rho^*(\mathcal{E})$ on $\overline{E}$ is coherent, hence generated by its global sections since $\overline{E}$ is Stein. Let $W\subset H^0(\overline{E},\rho^*(\mathcal{E}))$ be a vector subspace of dimension $g$ which generates $\rho^*(\mathcal{E})$ at the generic point of $\overline{E}$, and let $ev:W\times \overline{E}\cong \C^{2g}\to V(E')$ be the resulting meromorphic and bimeromorphic map, obtained from the injection $\pi'^*:H^0(\overline{E},\rho^*(\mathcal{E}))\to H^0(E'\times_E\overline{E},V(E'))$.

We thus obtain a dominating meromorphic map $\C^{2g}\to S^{[g]}$ by composing $ev$ with the bimeromorphic maps between $\overline{J^0}(E'),\overline{J^g}(E'),S^{[g]}(E')$, and finally projecting $S^{[g]}(E')$ onto $S^{[g]}$.

This completes the proof of Theorem \ref{C^2-dom K3}.\end{proof}

%% file: secthyp.tex
\section{Hyperbolicity of symmetric products}

\subsection{A remark on the Kobayashi pseudometric}

For any (irreducible) complex space $Z$, let $d_Z$ be its Kobayashi pseudo-distance. We say that $Z$ is generically hyperbolic if $d_Z$ is a metric on some nonempty Zariski open subset of $Z$. 

\begin{question} Assume $X$ is smooth, compact and generically Kobayashi hyperbolic with $n>1$. Is then $X_m$ is generically Kobayashi hyperbolic for any $m>0$?
\end{question}

One remark in this context. Let $(X^m)^*\subset X^m$ be the Zariski open subset consisting of ordered $m$-tuples of distinct points of $X$. The complement of $(X^m)^*$ has codimension $n\geq 2$ in $X^m$. By \cite{Kobayashi} Theorem 3.2.22, $d_{X^m\vert (X^m)^*}=d_{(X^m)^*}$. Let $q_m:X^m\to X_m$ be the natural quotient, and $X^*_m:=q_m((X^m)^*)$, so that $X^*_m$ has a complement of codimension $n$ in $X_m$ as well, which is the singular set of $X_m$. Moreover, $(X^m)^*=q_m^{-1}(X^*_m)$.
From \cite{Kobayashi} 3.1.9 and 3.2.8, we get: 

$d_{X^*_m}([x_1,\dots,x_m],[y_1,\dots,y_m])={\rm inf}_{s\in S_m}\{{\rm max}_{i=1,\dots,m}\{d_X(x_i,y_{s(i)})\}\}$. 

Although the complement $X_m^{sing}$ of $X^*_m$ in $X_m$ has codimension $n\geq 2$ (and the singularities are canonical quotient), it is not true that $d_{X_m\vert X^*_m}=d_{X^*_m}$ in general, as the following example shows. Even more, the pseudometric may degenerate away from $X_m^{sing}$, so the problem is not a local one near $X_m^{sing}$.

Let $C\subset X$ be an irreducible curve of geometric genus $g$ with normalisation $\hat{C}$ on $X$, and take $m\geq g$. Then $\hat{C}_m\to {\rm Alb}(C)$ is a surjective morphism with generic fibres $\P_{m-g}$, and there is then a natural generically injective map from $\hat{C}_m$ to $X_m$ showing that $d_{X_m}$ vanishes identically on its image.

If the answer to the above question is affirmative (as it should be if and only if $X$ is of general type, after S. Lang's conjectures), the vanishing locus of $d_{X_m}$ appears to have an involved structure. In particular, it should contain the union of all the $\hat{C})_m$ whenever $g(\hat{C})\leq m$, and this union should not be Zariski dense.

\begin{example} The simplest possible example might be a surface $S:=C\times C'$, where $C,C'$ are smooth projective curves of genus $2$, and $m=2$. In this case, the natural map $S_2\to C_2\times C'_2$ is a ramified cover of degree $2$ branched over $R:=(2C)\times C'_2\cup C_2\times (2C')$, where $(2C)\subset C_2$ is the divisor of double points (and similarly for $(2C'))$. Notice that $C_2$ identifies naturally with the $Pic_2(C)$, the Picard variety of line bundles of degree $2$ on $C$, isomorphic to $Jac(C)$, blown-up over the point $\{K_C\}$, and $2C$ embeds $C$ in $C_2$, its image meeting the exceptional divisor of $C_2$ in the $6$ ramification points of the map $C\to \Bbb P_1$ given by the linear system $\vert K_C\vert$. Thus $2C\subset C_2$ is an ample divisor (similarly for $C')$.
\end{example}

As a first step towards the above question, let us show the following result which in particular implies that entire curves in the above example cannot be Zariski dense.

\begin{proposition}
Let $X$ be a complex projective variety of dimension $n$ with irregularity $q:=h^0(X,\Omega_X)$. 
\begin{enumerate}
\item If $m \cdot n < q$ then entire curves  in $X_m$ are not Zariski dense.
\item If $X$ is of general type, $n\geq 2$ and $m \cdot n \leq q$ then entire curves  in $X_m$ are not Zariski dense.
\end{enumerate}
\end{proposition}

\begin{proof}
Let $\alpha: X \to A$ be the Albanese map. It induces the Albanese map $\alpha_m: X_m \to A$.
If $\dim X_m=m \cdot n < q=\dim A$ then by the classical Bloch-Ochiai's Theorem, entire curves  in $X_m$ are not Zariski dense.
If $X$ is of general type, by \cite{AA02} $X_m$ is of general type. Therefore by Corollary 3.1.14 \cite{Yam04}, if $\dim X_m=m \cdot n \leq q=\dim A$, entire curves  in $X_m$ are not Zariski dense.
\end{proof}

\subsection{Jet differentials on resolutions of quotient singularities} \label{sectjetdifforb}

We recall here some basic definitions related, on the one hand, to natural orbifold structures on resolution of quotients singularities (see \cite{CRT17, cad18, CDG19}), and on the other hand, to orbifold jet differentials (see \cite{CDR18}). The basic result we will need is given by Proposition \ref{proporbjet}.
\medskip

\subsubsection{Jet differentials on orbifolds} Let us give some details about the very basic notion of orbifold jet differentials that we will use in the following. For our purposes, it will be enough to consider only orbifolds of the form $(X, \Delta = \sum_{i} (1 - \frac{1}{m_i}) D_i)$, with $m_i \in \mathbb N_{\geq 1}$. Also, rather than using the {\em geometric orbifold jet differentials} defined in \cite{CDR18},  it will also suffice to consider jet differentials adapted to {\em divisible holomorphic curves} in the sense of [{\em loc. cit.}, Definition 1.1]. The latter jet differentials admit a very simple description. For any $k, r \in \mathbb N$, we will denote by $E_{k, r}^{GG} \Omega_{(X, \Delta)}^{\rm div}$ the vector bundle of divisible orbifold jet differentials of order $k$ and degree $r$, whose sections in orbifold local charts adapted to $\Delta$ can be described as follows. Assume that $(t_1, ..., t_{p}, t_{p+1}, ..., t_n) \in U \longmapsto (t_1^{m_1}, ..., t_p^{m_p}, t_{p+1}, ..., t_n) \in V$ is such a chart. Then, the local sections of $E_{k, r}^{GG} \Omega_{(X, \Delta)}^{\rm div}$ corresponds to the regular sections of $E_{k, r}^{GG} \Omega_U$ on $U$, which are invariant under the deck transform group. Remark that we could also have defined $E_{k, r}^{GG} \Omega_{(X, \Delta)}^{\rm div}$ in terms of a global \emph{adapted covering} instead of local orbifold charts.

\subsubsection{Natural orbifold structure on resolutions of a quotient singularities} Consider now a quotient $Y = \quotient{G}{X}$ where $X$ is smooth, and $G$ finite. If $\widetilde{Y} \longrightarrow Y$ is a resolution of singularities, we can endow it with a natural orbifold structure, by assigning to every exceptional divisor $E \subset \widetilde{Y}$ the rational multiplicity $1 - \frac{1}{m}$, where $m$ is the order of the element $\gamma \in G$ associated with the meridional loop around the generic point of $E$ (see \cite{CDG19, cad18}).
\medskip

With these notations, the following proposition is then essentially tautological.

\begin{proposition} \label{proporbjet}
	Let $X$ be a complex manifold, and let $G \subset \mathrm{Aut}(X)$ be a finite subgroup. Let $p : X \longrightarrow Y = \quotient{G}{X}$ be the quotient map, and $\widetilde{Y} \overset{\pi}{\longrightarrow} Y$ be a resolution of singularities. Let $(\widetilde{Y}, \Delta)$ be the natural orbifold structure on $\widetilde{Y}$.  Let $A$ be a $G$-invariant divisor on $X$, and $B = p_\ast A$ the associated Cartier divisor on $Y$.

	For $k, r \in \mathbb N$, we let $\sigma \in H^0(X, E_{k,r}^{GG} \Omega_X \otimes \mathcal O(-A))$ be a $G$-invariant section. Then $\pi^\ast p_\ast \sigma$ induces an element of $H^0( \widetilde{Y}, E_{k,r}^{GG} \Omega_{(\widetilde{Y}, \Delta)}^{\rm div} \otimes \mathcal O(- \pi^\ast B))$. 
\end{proposition}

\begin{remark}
	With the notations of the previous proposition, we see that if $r$ is divisible enough, and if $f$ is a local section of $\mathcal O_{\widetilde{Y}}(-r \Delta) \subset \mathcal O_{\widetilde{Y}}$, then $f \cdot \pi^\ast p_\ast \sigma$ is a holomorphic section of $E_{k, r}^{GG} \Omega_{\widetilde{Y}} \otimes \mathcal O(- \pi^\ast B)$.
\end{remark}

\subsection{A first criterion for the hyperbolicity of symmetric products} \label{sectcrithyp}

Before presenting our next hyperbolicity result, let us first prove a proposition that will allow us later on to compensate for the divergence of natural orbifold objects on resolutions of $X_m$. We resume the notations introduced in Section \ref{sectnot}.

\begin{proposition} \label{propbaselocus}
	Let $X$ be a complex projective manifold, and let $A$ be a very ample divisor on $X$. Let $\pi : \widetilde{X}_m \longrightarrow X_m$ be a log-resolution of singularities, and let $\Delta$ be the exceptional divisor with its reduced structure. Then
	$$
	\mathbb{B} (\pi^\ast A_\flat - \frac{1}{2(m-1)} \Delta) \subset |\Delta|,
	$$
	where $\mathbb B$ denotes the stable base locus.
\end{proposition}

We break the proof of this proposition into several lemmas.

\begin{lemma} \label{lemvanishingorder} Let $U$ be a complex manifold, let $G \subset \mathrm{Aut}(U)$ be a finite group, and let $p : U \longrightarrow \quotient{G}{U} = V$ be the quotient map. Let $A$ be a divisor on $X$, and let $A^\sharp = \sum\limits_{\gamma \in G} \gamma^\ast A$ and $A_\flat = p_\ast A^\sharp$. Note that $A^\sharp$ is $G$-invariant, and $A_\flat$ is Cartier on $V$. Let $W \subset U$ be an irreducible component of the subset of points stabilized by some element of $G$. Let $s \in \Gamma (U, A^\sharp)$ be a $G$-invariant section vanishing at order $r$ along $W$, for some $r \geq 1$. Then, we have the following.
\begin{enumerate} \itemsep=0em
		\item $s$ descends to a section $\sigma \in \Gamma(X, A_{\flat})$ ;
		\item let $\widetilde{X} \overset{\pi}{\longrightarrow} X$ be a resolution of singularities, and let $E \subset \widetilde{X}$ be an exceptional divisor such that $\pi(E) \subset p(W)$. Let $m$ be the multiplicity of $E$ for the natural orbifold structure on $\widetilde{X}$. Then, $\pi^\ast \sigma$, seen as a section of $\pi^\ast A_\flat$, vanishes at order $\geq \frac{r}{m}$ along $E$.	
\end{enumerate}
\end{lemma}
\begin{proof}
	(1) is trivial. Let us prove (2). Let $H \subset G$ be the stabilizer of the generic point of $\pi(E)$. By definition of $A^\sharp$, we may find an $H$-invariant trivialization $e$ of $A^\sharp$ near this generic point. Besides, $s = f \, {e}$ for some $H$-invariant holomorphic function $f$ vanishing at order $r$ along $W$. Consider a polydisk $D \cong \Delta^n$ centered around a generic point of $E$, and let $D'$ be the normalization of the fibered product of $D$ and $U$ over $V$. We obtain the following diagram:
	\begin{center}
\begin{tikzcd}
 	& D' \cong \Delta \times \Delta^{n-1} \arrow[r, "\pi'"] \arrow[d, "p'"] & U \arrow[d, "p"] \\
	\left( \Delta^n \cap E \right) =  \{0\} \times \Delta^{n-1} \arrow[r, hookrightarrow] &D \cong \Delta \times \Delta^{n-1} \arrow[r, "\pi"]  & V
\end{tikzcd}
	\end{center}
	Since $f$ is $H$-invariant, $f \circ \pi' = f' \circ p'$ for some holomorphic function $f'$ on $D \cong \Delta \times \Delta^{n-1}$. Moreover, we have $\sigma = f' \; e_{\flat}$, where $e_{\flat}$ is the section of $A_\flat$ induced by $e$. The holomorphic function $f$ vanishes at order $r > 0$ along $V$, so $f \circ \pi'$ vanishes at order $\geq r$ along $\{ 0 \} \times \Delta^{n-1}$. Since $p'(w, z) = (w^m, z)$, this implies that $f'$ vanishes at order $\geq \frac{r}{m} > 0$ along $\{ 0 \} \times \Delta^{n-1} \subset \Delta^n$. This ends the proof.
\end{proof}

\begin{lemma} \label{lemsect}
	Let $N, m \geq 1$. We define $V = \mathbb P^N \times ... \times \mathbb P^N$ to be a product of $m$ copies of $\mathbb P^N$. Let $\mathfrak{D} = \{ (z_1, ..., z_m) \in V \; | \; \exists i \neq j,\; x_i = x_j\}  \subset V$ be the diagonal locus. Let $A \subset P^N$ be a hyperplane section, and $A^\sharp = \sum\limits_{i=1}^m {\rm pr_i^{\ast}} A$. 

	Then, for any $z \in V \setminus \mathfrak{D}$, there exists a $\mathfrak{S}_m$-invariant section $$s \in \Gamma(V, \mathcal O_V( 2(m-1)  A^\sharp)),$$ with $s(x) \neq 0$, and such that $s$ vanishes at order $2$ along $\mathfrak{D}$. 
\end{lemma}
\begin{proof}
	Let $z = (z_1, ..., z_m) \in V \setminus \mathfrak D$. Write $(\mathbb P^N)_i$ to denote the $i$-th factor of $V$. For any $i < j$, we have $z_i \neq z_j$, so for two generic hyperplane linear sections $X, Y \in |A|$, we have
	\begin{equation} \label{eqnonzero}
		X (z_i) Y (z_j) - X (z_j) Y (z_i) \neq 0.
	\end{equation}
	Indeed, we can choose $X$, $Y$ so that $X(z_i) \neq 0$ and $X(z_j) = 0$ (resp. $Y(z_i) = 0$ and $Y(z_j) \neq 0$).

	Now, for each $1 \leq i \leq m$, choose two generic linear forms $X_i$ and $Y_i$ on $(\mathbb P^N)_i$. We let
	$$
	s = \prod_{i < j} \left( X_i Y_j - X_j Y_i \right)^2
	$$
	This is a section of $\bigotimes_{i = 1}^m p_i^\ast \mathcal O(2(m-1)) = \mathcal O(2(m-1)A^\sharp)$. By the argument above, we have $s(z) \neq 0$, and $s$ vanishes on $\mathfrak{D}$ at order $2$  by Lemma \ref{lemcomputorder}. We check that $s$ is invariant under all transpositions $(i \,j) \in \mathfrak{S}_m$. This proves that $s$ is $\mathfrak{S}_m$-invariant. 
\end{proof}

\begin{lemma} \label{lemcomputorder}
	Let $X_1$, $Y_1$, $X_2$, $Y_2$ be generic hyperplane sections on $\mathbb P^N$. Then the homogeneous polynomial $X_1 Y_2 - X_2 Y_1$ vanishes at order $1$ along the diagonal of $\mathbb P^N \times \mathbb P^N$.
\end{lemma}
\begin{proof}
	We let $2 u = X_1 + X_2$, $2 v = X_1 - X_2$ (resp. $2 u' = Y_1 + Y_2$, $2 v' = Y_1 - Y_2$). Then, we can write
	\begin{align*}
		X_1 Y_2 - X_2 Y_1 & = (u + v)(u' - v') - (u - v)(u' + v') \\ 
				  & = - 2 u v' + 2 u' v. 
	\end{align*}
	This expression is of degree $1$ in $v'$ and $v$, so for generic $u, u'$, it vanishes at order one along the diagonal.
\end{proof}

The proof of Proposition \ref{propbaselocus} is now straightforward.

\begin{proof}[Proof of Proposition \ref{propbaselocus}]

	Let $x \in \widetilde{X}_m \setminus |\Delta|$, and let $x_0 \in X^m$ be such that $p(x_0) = \pi (x)$. Since $x$ is not in $|\Delta|$, $x_0$ is not in the diagonal locus of $X^m$. Using the embedding $X \subset \mathbb P^N$ provided by the very ample divisor $A$, Lemma \ref{lemsect} gives a $\mathfrak{S}_m$-invariant section $\sigma \in H^0( X^m, 2(m-1) A^\sharp)$ such that $\sigma(x_0) \neq 0$, and such that $\sigma$ vanishes at order $2$ along the diagonal locus.

	Applying now Lemma \ref{lemvanishingorder} to $\sigma$, we see that the induced section $$\pi^\ast p_\ast \sigma \in H^0(\widetilde{X}_m, 2(m-1) \pi^\ast A_\flat)$$ vanishes along $|\Delta|$. Moreover, we have $\pi^\ast p_\ast \sigma(x) \neq 0$, which gives the result.
\end{proof}
\medskip

We are ready to state our hyperbolicity criterion, in terms of the existence of sufficiently many jet differentials of bounded order on $X$. Again, we refer to \cite{dem12a} for the basic definitions related to jet differentials. Let us simply recall that the locus of singular jets $X_k^{GG, {\rm sing}} \subset X_k^{GG}$ is the subset of all classes of $k$-jets $[f : \Delta \to X]_k$ such that $f'(0) = 0$. Also, if $V \subset H^0(X, E_{k,r}^{GG} \Omega_X)$ is a vector subspace, then ${\rm Bs}(V) \subset X_k^{GG}$ is the subsets of classes of $k$-jets solutions to every equation in $V$.
\medskip

\begin{theorem} \label{thmhypcrit}
	Let $X$ be a complex projective manifold. Let $A$ be a very ample line bundle on $X$. Let $Z \subset X$, and $k, r ,d \in \mathbb N^\ast$. We make the following hypotheses.

\begin{enumerate}\itemsep=0.5em
		\item Assume that
			$$
			\mathrm{Bs} \left(H^0 (X, E_{k,r}^{GG} \Omega_X \otimes \mathcal O(- d A)) \right) \subset X_k^{GG, \mathrm{sing}} \cup \pi_k^{-1}(Z).
			$$

		\item Assume that $\frac{d}{r} > 2 m(m-1)$.
\end{enumerate}
	Then, $\mathrm{Exc}(\widetilde{X}_m) \subset |\Delta| \cup \pi^{-1}(\mathfrak{d}_1(Z))$.
\end{theorem}

\begin{proof}
	Let $f : \mathbb C \longrightarrow \widetilde{X}_m$ be an entire curve such that $f(\mathbb C) \not\subset | \Delta |$. Let $U = \mathbb C - f^{-1}(|\Delta|)$, and, as before $\mathfrak D = \bigcup\limits_{i \neq j} \{ x_i = x_j\} \subset X^m$ . We consider the following diagram:
	\begin{center}
\begin{tikzcd}
	\widetilde{U}  \arrow[r, "g"] \arrow[d, "q"] & X^m \setminus \mathfrak{D}   \arrow[d, "p"] \arrow[r, "\mathrm{pr}_i"] & X \\
	U  \arrow[r, "f"] & (X_m)_{\mathrm{reg}}  & 
\end{tikzcd}
	\end{center}
where $q$ is the universal covering map, and $g$ is an arbitrary lift of $f$. Without loss of generality, we can assume that all $\mathrm{pr}_i \circ g$ are non-constant ($1 \leq i \leq m$). Indeed, if one of these maps is constant, it suffices to replace $X^m$ (resp. $X_m$) by the product $Y = X \times ... \times X$ over a number $m' < m$ of factors (resp. by $X_{m'} = \quotient{{\mathfrak S}_{m'}}{Y}$).

\medskip

	We may assume that $\mathrm{Im}(\mathrm{pr}_i \circ g) \not\subset Z$ for all $1 \leq i \leq m$, otherwise the proof is finished. Thus, there exists $t \in \widetilde{U}$ such that $(\mathrm{pr}_i \circ g)(t) \not\in Z$, and $(\mathrm{pr}_i \circ g)'(t) \neq 0$ for all $1 \leq i \leq m$. By the hypothesis (1), there exists $\sigma \in H^0 (X, E_{k,m}^{GG} \Omega_X \otimes \mathcal O(- d A))$ such that for all $1 \leq i \leq m$, we have $\sigma_{g(t)}\cdot (\mathrm{pr}_i \circ g) \neq 0$, and in particular
	$$
	\sigma( \mathrm{pr}_i \circ g) \not\equiv 0
	$$
	for all $i$.
\medskip

	Thus, $\sigma^{\sharp} \overset{\rm def}{=} \bigotimes_{i=1}^m \mathrm{pr}_i^\ast(\sigma)$ is a $\mathfrak{S}_m$-invariant jet differential in $H^0(X^m, E_{k, rm}^{GG} \Omega_X \otimes \mathcal O(- d A^\sharp))$ such that $\sigma^\sharp(g) \not\equiv 0$. By Proposition \ref{proporbjet}, $\sigma^\sharp$ induces a section 
	$$
	\sigma_\flat \in H^0(\widetilde{X}_m, E_{k, rm}^{GG} \Omega_{(\widetilde{X}_m, \Delta)}^{\rm div} \otimes \mathcal O( - d \pi^\ast A_{\flat})).
	$$
	We have moreover $\sigma_\flat(f) \not\equiv 0$.
	\medskip

	Now, by Proposition \ref{propbaselocus}, for $a \geq 1$ divisible enough, there exists $s \in H^0 (\widetilde{X}_m, a ( \pi^\ast A_\flat - \frac{1}{2(m-1)} \Delta))$ such that $s|_{f(\mathbb C)} \not\equiv 0$. Thus, by the remark following Proposition \ref{proporbjet}, $s^{2rm(m-1)} \sigma_{\flat}^a$ induces a non-orbifold section 
	$$
	\sigma' \in H^0 \left(\widetilde{X}_m, E_{k, arm}^{GG} \Omega_{\widetilde{X}_m} \otimes \mathcal O\left(a(2rm(m-1) - d) \pi^\ast A_{\flat} \right)\right),
	$$
	and $\sigma'(f) \not\equiv 0$. Since $2rm(m-1) < d$ and $\pi^\ast A_{\flat}$ is big on $\widetilde{X}_m$, this is absurd by the fundamental vanishing theorem of Demailly-Siu-Yeung (see \cite{dem12a}).
\end{proof}

\subsection{Applications}

\subsubsection{Hypersurfaces of large degree}

Using Theorem \ref{thmhypcrit}, we can now obtain hyperbolicity results for the varieties $X_m$ when $X \subset \mathbb P^{n+1}$ is a generic hypersurface of large degree. To do this, we will make use of several important recent results concerning the base loci of jet differentials on such hypersurfaces. Let us begin with the algebraic degeneracy of entire curves.
\medskip

The recent work of B\'erczi and Kirwan \cite{BK19} gives new effective degrees for which a generic hypersurface has enough jet differentials to ensure the degeneracy of entire curves. This improvement of \cite{DMR10} yields the following result.

\begin{theorem}[\cite{BK19}] \label{thmalgdeg} Let $X \subset \mathbb P^{n+1}$ be a generic hypersurface of degree
	$$
	d \geq 16 n^5 (5n + 4).
	$$
	Then, if $r \gg 0$ is divisible enough, we have
	\begin{equation}\label{eqinclBK}
		\mathrm{Bs} \left[ H^0(X, E_{n,r}^{GG} \Omega_X \otimes \mathcal O ( -r \frac{d - n - 2}{16n^5} + r (5n+3) )  ) \right] \subset X_{k}^{GG, \mathrm{sing}} \cup \pi_k^{-1}(Z)
	\end{equation}
	for some algebraic subset $Z \subsetneq X$.
\end{theorem}
\begin{remark}
	As explained in \cite{BK19}, the coefficient $5n+3$ comes from Darondeau's improvements \cite{dar16} for the pole order of slanted vector fields on the universal hypersurface. It seems to us by reading \cite{dar16} that we should actually expect the slightly better value $5n - 2$.
\end{remark}

	We deduce immediately from Theorem \ref{thmhypcrit} the following consequence of this result. 

\begin{corollary}
	Let $m, n \in \mathbb N^\ast$. Let $X \subset \mathbb P^{n + 1}$ be a generic hypersurface of degree 
	$$
	d \geq 16n^5 (5n + 2m^2 + 4).
	$$
	Then there exists $Z \subsetneq X$ such that $\mathrm{Exc}(X_m) \subset \mathfrak{d}_1 (Z)$. 
\end{corollary}
\begin{proof}
	Because of \eqref{eqinclBK}, the conditions of Theorem \ref{thmhypcrit} will be satisfied if
	$$
	\left( \frac{d - n - 2}{16n^5} - (5n + 3) \right) > 2m (m-1),
	$$
	which is implied by our hypothesis. We have then $\mathrm{Exc}(X_m) \subset (X_m)_{\mathrm{sing}} \cup \mathfrak{d}_1(Z)$ for some $Z \subsetneq X$. Since $(X_m)_{\mathrm{sing}}$ is a union of $X_{m'}$ for $m' < m$, an induction on $m$ permits to conclude.
\end{proof}

It is also possible to obtain the hyperbolicity of $X_m$ when $X$ has large enough degree, using all the recent work around the Kobayashi conjecture (cf. \cite{bro17, den17, dem18, RY18}). The main result of \cite{RY18} permits to reduce the proof of the hyperbolicity of $X$ to results such as Theorem \ref{thmalgdeg}, and gives in particular the following. 

\begin{theorem}[\cite{RY18}] Let $d, n, c, p \in \mathbb N$. Suppose that for a generic hypersurface $X' \subset \mathbb P^{n+1+p}$ of degree $d$, we have
	$$
	\mathrm{Bs} \left( H^0(X', E_{k,r}^{GG} \Omega_{X'} \otimes \mathcal O(-1)) \right) \subset X_k'{}^{GG, \mathrm{sing}} \cup \pi_k^{-1} (Z'),
	$$
	for some algebraic subset $Z' \subset X'$ satisfying $\mathrm{codim}(Z') \geq c$. Then, for a generic hypersurface $X \subset \mathbb P^{n+1}$ of degree $d$, we have
	$$
	\mathrm{Bs} \left( H^0(X, E_{k,r}^{GG} \Omega_{X} \otimes \mathcal O(-1)) \right) \subset X_k^{GG, \mathrm{sing}} \cup \pi_k^{-1} (Z),
	$$
	for some subset $Z \subset X$ with $\mathrm{codim}(Z) \geq c + p$.
\end{theorem}

Letting $d = n - 1$, we deduce from this and Theorem \ref{thmhypcrit}, combined with Theorem \ref{thmalgdeg}:

\begin{corollary} \label{corolhypsurf}
	Let $X \subset \mathbb P^{n+1}$ be a generic hypersurface of degree
	$$
	d \geq (2n - 1)^5 (2m^2 + 10 n - 1).
	$$
	Then $X_m$ is hyperbolic.
\end{corollary}

\subsubsection{Complete intersections of large degree}

We can also obtain a hyperbolicity result for symmetric products of generic complete intersections of large multidegree, using the work of Brotbek-Darondeau and Xie on Debarre's conjecture (see \cite{BD18, xie18}). The effective bound in the theorem below is provided by \cite{xie18}.

\begin{theorem}[\cite{BD18, xie18}]
	Let $n, n', d \geq 1$, and assume that $n' \geq n$. Let $X \subset \mathbb P^{n+n'}$ be a complete intersection of multidegrees
	$$
	d_1, ..., d_{n'} \geq (n+n')^{(n+n')^{2}} \cdot d
	$$
	Then $\Omega_X \otimes \mathcal O(-d)$ is ample. In particular
	$$
	\mathrm{Bs}( H^0(X, E_{1, r}^{GG} \otimes \mathcal O(-rd)) = \emptyset
	$$
	for $r \gg 1$.
\end{theorem}

By Theorem \ref{thmhypcrit} and the same induction argument on $m$ as above, the following corollary is immediate.

\begin{corollary} \label{corolintcomp}
	Let $m, n \in \mathbb N^\ast$ and let $n' \geq n$. Let $X \subset \mathbb P^{n + n'}$ be a generic complete intersection of multidegrees 
	$$d_1, ..., d_{n_1} > (n+n')^{(n+n')^2} (2m(m-1))$$

	Then $X_m$ is hyperbolic.
\end{corollary}

\begin{remark} For $d_1$ large enough, Corollary \ref{corolintcomp} is trivially implied by Corollary \ref{corolhypsurf}. Indeed, if $X \subset H$, where $H$ is a degree $d_1$ hypersurface, $X_m$ embeds in $H_m$. 
\end{remark}

\subsection{Higher dimensional subvarieties} 
In this section, we gather several results related to the subvarieties of $X_m$, when $X$ is a "sufficiently hyperbolic" manifold. In particular, when $\Omega_X$ is ample, we will show that a generic subvariety of $X_m$ of codimension higher than $n-1$ is of general type (see Theorem \ref{thmsubv}).

\medskip

\begin{lemma} \label{lemcrit}
	Assume that $X$ is a complex manifold of dimension $n$, with $n \geq 2$, and let $\mathfrak{S}_m$ act on $X^m$. Let $\alpha \in [0, 1]$. If 
	$$d \geq n(m-1) + 2 - \alpha \cfrac{(n-2)(m-2)}{2},$$ then the condition $(I'_{x, d, \alpha})$ of Section \ref{sectcrit} is satisfied for every $x \in X^m$. In particular, if $d \geq n(m-1) + 2$, then the condition $(I_{x, d})$ is satisfied for any $x \in X^m$.
\end{lemma}
\begin{proof}

	Let $\sigma \in \mathfrak{S}_m \setminus \{ 1 \}$, and let $\sigma = \sigma_1 ... \sigma_t$ be a decomposition of $\sigma$ into cycles with disjoint supports. For each $\sigma_i$, let $r_i = \mathrm{ord}(\sigma_i)$, and assume that $r_1 \geq ... \geq r_l > 1$, and $r_{l+1} = ... = r_{l + s} = 1$, with $s = t - l$. Then, the order of $\sigma$ is $r = \mathrm{lcm}(r_1, ..., r_l)$, and the $a_i$ appearing in condition $(I_{x, d})$ are the integers $j \frac{r}{r_k}$ $(1 \leq k \leq s, 0 \leq j < r_k)$, each one repeated $n$ times. We see in particular that $0$ appears with multiplicity $nt = n (s + l)$, and that each non-zero $a_i$ is larger than $\frac{r}{\max\limits_{1 \leq j \leq l} r_j}$.

	We need to check that for any choice of $d$ distinct elements $a_{i_1}, ..., a_{i_d}$ among the $a_i$, the sum is larger than $(1 - \alpha)r$. The lowest possible sum is reached when all the $0$ appear in it. Thus, the sum of the $a_{i_j}$ is larger than
	$$
	( d - n(s + l)) \frac{r}{\max\limits_{1 \leq j \leq l} r_j}.
	$$

	The last quantity is larger than $r(1 - \alpha)$ if the following inequality is satisfied:
	\begin{equation} \label{ineqcond}
		n (s + l) + (1 - \alpha) \max\limits_{1 \leq j \leq l} r_j \leq d
	\end{equation}

	Now, we have $\max\limits_{1 \leq j \leq l} r_j \leq \sum\limits_{1 \leq j \leq l} r_j = m - s$, and $2l + s \leq \sum_{1 \leq j \leq l} r_j + s =  m$ hence $l \leq \frac{m-s}{2}$. Putting everything together, we see that the following is always satisfied:
	\begin{align*}
		n(s + l) + \max\limits_{1 \leq j \leq l} r_j & \leq  \left( \frac{n}{2} + 1 \right) m + (1 - \alpha) \left( \frac{n}{2} - 1 \right)s. 
	\end{align*}
	Since $n \geq 2$ and $1 - \alpha \geq 0$, the right hand side is maximal if $s$ is maximal, equal to $m - 2$; this right hand side is then equal to $n(m-1) + 2 - \alpha \cfrac{(n-2)(m-2)}{2}$ (thus the maximum is reached for $r_1 = 2, r_2 = ... = r_t = 1$, i.e. when $\sigma$ is a transposition). Thus, if $d \geq n(m-1) + 2 - \alpha \cfrac{(n-2)(m-2)}{2}$, the inequality \eqref{ineqcond} is satisfied, which gives the result. 
\end{proof}

In the next definition, we state a condition that will later imply that a generic subvariety of $X_m$ of high enough dimension is of general type (see Theorem \ref{thmsubv}).

\begin{defi}
	Let $X$ be a complex projective manifold, let $\Sigma \subsetneq X$ be a proper algebraic subset, and let $A$ be an effective divisor on $X$. We say that $X$ \emph{satisfies the property $(H_{\Sigma, A})$}, if the following holds. 
	
	Let $V \subset X$ be a subvariety of arbitrary dimension $d$, not included in $\Sigma$ and $A$. Then, there exists $q, r \geq 1$, and a section $\sigma \in H^0 (X, ( \bigwedge^d \Omega_X )^{\otimes q})$, with \emph{non-zero} restriction
	$$
	\sigma|_{(\bigwedge^d T_{V^{\rm reg}})^{\otimes q}} \in H^0 (V^{\mathrm{reg}}, ( \bigwedge^d \Omega_V )^{\otimes q} \otimes \mathcal O(- r A|_V)) - \{ 0 \}.
	$$
\end{defi}
\medskip

Under suitable positivity hypotheses on the cotangent bundle of a complex manifold, it is not hard to check that the previous condition is satisfied, as we will show in the next proposition.

Recall that if $E \longrightarrow X$ is a vector bundle, its {\em augmented base locus} is the algebraic subset $\mathbb{B}_+(E) \subset X$ defined as follows. Let $p : \mathbb P(E) \longrightarrow X$ be projectivized bundle of rank one quotients of $E$, and $\mathcal O(1)$ be the tautological line bundle on $\mathbb P(E)$. Then, if $A$ is any ample line bundle on $X$, we let $$\mathbb{B}_+(E) = p( \mathbb{B_+}(\mathcal O(1))),$$ where $\mathbb B_+ (\mathcal O(1)) = \bigcap\limits_{l \geq 1} {\rm Bs}(\mathcal O(l) \otimes p^\ast A^{-1})$. The \emph{ample locus} of $E$ is the (possibly empty) open subset $X \setminus \mathbb B_+(E)$.

\begin{proposition} \label{propbigample}
	Let $X$ be a complex projective manifold such that $\Omega_X$ is big. Let $A$ be any very ample divisor on $X$.
	
	\begin{enumerate} \itemsep=0em
			\item if $\mathbb B_+(\Omega_X) \neq X$, then $X$ satisfies the property $(H_{\mathbb B_+(\Omega_X), A})$;

	\item	if $\Omega_X$ is ample, then $X$ satisfies the property $(H_{\emptyset, A})$.
\end{enumerate}
\end{proposition}
\begin{proof}
	(1) Let $V \subset X$ be a $d$-dimensional subvariety such that $V \not\subset \mathbb{B}_+(\Omega_X)$ and $V \not\subset A$. By general properties of ampleness of vector bundles, we have the inclusion $\mathbb{B}_+(\bigwedge^d \Omega_X) \subset \mathbb{B}_+(\Omega_X)$ (this can be seen easily e.g. from \cite[Corollary 6.1.16]{lazpos2})

	Thus, if $x \in V \setminus \mathbb{B}_+(\bigwedge^d \Omega_X)$ is a smooth point of $V$, and $w = \bigwedge^d T_{V, x}$, there exists $\sigma \in H^0(X, S^m(\bigwedge^d \Omega_X) \otimes \mathcal O(-A))$ such that $\sigma_x(w^{\otimes m}) \neq 0$. In particular, since $\sigma$ vanishes along $A$, the restriction $\sigma|_{V}$ vanishes along $A \cap V$. The section $\sigma$ satisfies our requirements.
	\medskip
	
	(2) If $\Omega_X$ is ample, we have $\mathbb{B}_+(\Omega_X) = \emptyset$, so the result comes from the first point.
\end{proof}

In the next proposition, we show that the property $(H_{\Sigma, A})$ is stable under products.

\begin{proposition} \label{propstabproduct}
	Let $X_i$ $(i = 1, 2)$ be complex projective manifolds, and denote by $p_1, p_2 : X_1 \times X_2 \longrightarrow X$ the canonical projections. Assume that each $X_i$ satisfies the property $(H_{\Sigma_i, A_i})$ for some subvariety $\Sigma_i \subsetneq X_i$ and some divisor $A_i$ on $X$.
	
	Then $X_1 \times X_2$ satisfies the property $(H_{\Sigma, A})$, where $\Sigma = p_1^{-1}(\Sigma_1) \cup p_2^{-1}(\Sigma_2)$, and $A = p_1^\ast A_1 + p_2^\ast A_2$.

\end{proposition}

\begin{proof}
	Let $V \subset X_1 \times X_2$ be a $d$-dimensional subvariety such that $V \not\subset \Sigma$.  Let $d_2 = \dim p_2 (V)$, and let $d_1$ be the dimension of the generic fiber of $p_2 : V \longrightarrow p_2(V)$. We have $d_1 + d_2 = d$.
	\medskip

	(1) We deal first with the case $d_2 = 0$. Then, we have $\dim p_1(V) = d$, and $p_1(V) \not\subset \Sigma_1$ because $V \not\subset \Sigma$. Since $X_1$ satisfies $(H_{\Sigma_1})$, there exists integers $q, r \geq 1$, and a section $\sigma \in H^0(X_1, (\bigwedge^d \Omega_{X_1})^{\otimes q})$ such that $\sigma|_{\bigwedge^d T_{p_1(V)^{\mathrm{reg}}}}$ vanishes at order $r$ along $A_1$. Thus, $(p_1)^\ast \sigma \in H^0(X_1 \times X_2, (\bigwedge^d \Omega_{X_1})^{\otimes q})$. We also have $(p_1)^\ast \sigma|_{\bigwedge^d T_{V^{\mathrm{reg}}}} \not\equiv 0$, and this section vanishes at order $r$ along $p_1^\ast A_1 + p_2^\ast A_2|_{V} = p_1^\ast A_1|_{V} $. This ends the proof in this case.
	\vspace{1em}

	(2) Assume now that $d_2 > 0$. Let $x_2 \in X_2$ be generic so that $\dim (V_{x_2}) = d_1$ and $p_1 (V_{x_2}) \not\subset \Sigma_1$, where $V_{x_2} = p_2^{-1}(x_2) \cap V$. Let $V_2 = p_2(V)$, and $V_1 = p_1(V_{x_2})$.
	
	For each $i$, we have $V_i \not\subset \Sigma_i$, so there exists integers $q_i, r_i \geq 1$, and a section $\sigma_i \in H^0( X_i, (\bigwedge^{d_i} \Omega_{X_i})^{\otimes d_i})$ whose restriction to $(\bigwedge^{d_i} T_{V_i^{\mathrm{reg}}})^{\otimes q_i}$ vanishes at order $r_i$ along $A_i|_{V_i}$. Then,
	$$
	\sigma = (p_1^\ast \sigma_1)^{\otimes q_2} \otimes (p_2^\ast \sigma_2)^{\otimes q_1}
	$$
	can be identified to a section in $H^0(X_1 \times X_2, ( \bigwedge^{d_1} p_1^\ast \Omega_{X_1} \otimes \bigwedge^{d_2} p_2^\ast \Omega_{X_2})^{\otimes q_1 q_2})$.  Since $\bigwedge^{d_1} p_1^\ast \Omega_{X_1} \otimes \bigwedge^{d_2} p_2^\ast \Omega_{X_2}$ is a direct factor of $\bigwedge^{d} \Omega_X \cong \bigwedge^{d_1 + d_2} ( p_1^\ast \Omega_{X_1} \oplus p_2^\ast \Omega_{X_2})$, we have obtained a section $\sigma \in H^0(X_1 \times X_2, (\bigwedge^d \Omega_{X_1 \times X_2} )^{\otimes q_1 q_2})$ which does not vanish along $V$.
	
	Moreover, by construction, the restriction of $\sigma$ to $(\bigwedge^d T_{V^{\mathrm{reg}}})^{\otimes q_1 q_2}$ vanishes along $B|_{V}$, where $B = q_2 r_1 \, p_1^\ast A_1 + q_1 r_2\, p_2^\ast A_2$. Since $q_2 r_1, q_1 r_2 > 0$, this restriction vanishes along $A$. This gives the result.
\end{proof}

In the case where $X_1 = X_2$, it is not hard to strengthen the property $(H_\Sigma)$ to obtain sections $\sigma$ invariant by permutation of $X_1$ and $X_2$. More precisely:

\begin{proposition}  \label{propsectinv}
	Let $X$ be a complex projective manifold satisfying the property $(H_{\Sigma, A})$ for some $\Sigma \subsetneq X$ and some ample divisor $A$ on $X$. Let $\Sigma' \subset X^m$ the subset of points with at least a coordinate in $\Sigma$. Let $\mathfrak{S}_m$ act on $X^m$ by permutation of the factors. Then, for any subvariety $V \subset X^m$ of dimension $d$ and such that $V \not\subset \Sigma'$, there exists an integer $q \geq 1$, and a $\mathfrak{S}_m$-invariant section $\sigma \in H^0(X^m, (\bigwedge^d \Omega_X)^{\otimes q} \otimes \mathcal O(-A^\sharp))^{\mathfrak{S}_m}$ such that $\sigma|_{\bigwedge^d T_{V^{\rm reg}}} \not\equiv 0$. 
	
	(Recall that $A^\sharp = \sum\limits_{i=1}^m {\rm pr}_i^\ast A$).
\end{proposition}
\begin{proof}
	By Proposition \ref{propstabproduct}, $X^m$ satisfies the property $(H_{\Sigma', A^{\sharp}})$ so there exists $q_0 \geq 1$ and a section $\sigma_0 \in H^0(X^m, (\bigwedge^d \Omega_X)^{\otimes q_0})$, such that $\sigma_0|_{(\bigwedge^d T_{V^{\rm reg}})^{\otimes q_0}}$ vanishes at order $r_0$ along $A^\sharp|_{V}$.

	Now, we let 
	$$\sigma = \bigotimes_{s \in \mathfrak S_m} s \cdot \sigma_0 \in H^0(X^m, (\bigwedge^d \Omega_X)^{\otimes m!\, q_0})$$
	The section $\sigma$ is $\mathfrak{S}_m$-invariant and vanishes along $A^\sharp$, hence satisfies our requirements.
\end{proof}

We now show the main hyperbolicity result of this section.

\begin{theorem} \label{thmsubv}
	Let $X$ be a complex projective manifold with $\dim X \geq 2$. Assume $X$ satisfies $(H_{\Sigma, A})$ for some $\Sigma \subsetneq X$ and some ample divisor $A$ on $X$. 
	
	Then, any subvariety $V \subseteq X_m$ such that ${\rm codim} V \leq n - 2$ and $V \not\subset  X_m^{\rm sing} \cup \mathfrak d_1(\Sigma)$ is of general type.
\end{theorem}

\begin{proof}
	Let $V \subset X_m$ be a $d$-dimensional variety satisfying the hypotheses above. We have then $d \geq (m-1)n + 2$. Let $X^m \overset{p}{\longrightarrow} X_m$ be the canonical projection. We do not lose generality in replacing $A$ by a high multiple (the condition $(H_{\Sigma, A})$ is preserved), and then moving it in its linear equivalence class, so we can assume that $V \not\subset |A|$.

	By Proposition \ref{propsectinv}, for $q \gg 0$, there exists a section $\sigma \in \Gamma( X^m, (\bigwedge^p \Omega_{X^m})^{\otimes q})^{\mathfrak{S_m}}$, whose restriction to $(\bigwedge^{d} T_{p^{-1}(V^{\rm reg})})^{\otimes q}$ vanishes along the $\mathfrak{S}_m$-invariant ample divisor $A^\sharp$. This section descends to $X_m$; moreover, for any resolution of singularities $\widetilde{X}_m$, Lemma \ref{lemcrit} shows that the Reid-Tai-Weissauer criterion of Proposition \ref{propextension} is appliable. Hence, $\sigma$ induces a section 
	$$\widetilde{\sigma} \in H^0 (\widetilde{X}_m, (\bigwedge^d \Omega_{\widetilde{X}_m})^{\otimes q}).$$  Moreover, the restriction of $\widetilde{\sigma}$ to $\bigwedge^{d} T_{V^{\rm reg}}$ vanishes on the ample divisor $A = p_\ast(A^\sharp)|_{V}$.

	Consider now a resolution of singularities $\widetilde{V} \overset{\varphi}{\longrightarrow} V$. The pullback $\varphi^\ast \sigma$ induces a section of $K_{\widetilde{V}}$ that vanishes on the big divisor $\varphi^\ast A$. This implies that $K_{\widetilde{V}}$ is big, so $V$ is of general type.
\end{proof}

\begin{remark}
	The bound on $\dim V$ in Theorem \ref{thmsubv} is sharp, as we can see from the following example. Let $C$ be a genus $2$ curve, and let $Y$ be any $(n-1)$-dimensional variety with $\Omega_Y$ ample. Let $X = C \times Y$. This manifold satisfies property $(H_{\emptyset, A})$ for some ample divisor $A$ by Propositions \ref{propbigample} and \ref{propstabproduct}.

\begin{enumerate} \itemsep=0em
		\item In the case $m = 2$: let $f : S^2 C \times Y \longrightarrow S^2 (C \times Y) =S^2 X$ be the generically injective map
			$$
			f( \, [c_1, c_2], y_1, ..., y_{n-1}) = [(c_1, y_1, ..., y_{n-1}), (c_2, y_1, ..., y_{n-1})].
			$$

			Since $g(C) = 2$, the variety $S^2(C)$ is birational to $\mathrm{Jac}(C)$ and thus $S^2 C \times Y$ is not of general type.

		\item In the case $m \geq 2$, consider the composition of $f \times \mathrm{Id}_{X^{m-2}}: S^2 C \times Y \times X^{m-2} \longrightarrow S^2 X \times X^{m-2}$ (where $f$ is as above) and of the natural map $g : S^2 X \times X^{m-2} \longrightarrow S^m X$.

			We have $\dim S^2 C \times Y \times X^{m-2} = n(m-1) + 1$, and the image $V = (g \circ f)(S^2 C \times Y \times X^{m-2})$ in $X_m$ is not of general type, since $S^2 C \times Y \times X^{m-2}$ is not.
\end{enumerate}

\end{remark}

Note that if the Green-Griffiths-Lang conjecture were true, then Theorem \ref{thmsubv} would imply the following result.

\begin{conjecture}
	Let $X$ be a complex projective manifold with $\Omega_X$ ample. Then, ${\rm codim\, Exc}(X_m) \geq n - 1$.
\end{conjecture}

We can use Theorem \ref{thmsubv} to prove the following weaker statement, that gives geometric restrictions on the exceptional locus on non-hyperbolic \emph{algebraic} curves in $X_m$.

\begin{corollary} \label{corolexcalg} Assume that $\Omega_X$ is ample. Then, there exist countably many proper algebraic subsets $V_k \subsetneq X_m$ $(k \in \mathbb N)$ containing the image of any non-hyperbolic algebraic curve. Moreover, the $V_k$ can be chosen so that for any component $W$ of $\mathfrak{D}_i(X_m)$ $(0 \leq i \leq n)$ containing $V_k$ ($k \in \mathbb N$), we have ${\rm codim}_{W} (V) \geq n - 1$.

	In particular (letting $i = 0$ and $W = X_m$), we have ${\rm codim}_{X_m} (V_k) \geq n - 1$ for all $k \in \mathbb N$.
\end{corollary}
\begin{proof}
	As the irreducible components of each $\mathfrak{D}_{i}(X_m)$ identify to copies of $X_{m - i}$, it suffices to prove the last claim, and to show the result for curves $C$ not included in $(X_m)_{\rm sing}$.

	 By \cite[Proposition 2.8]{kol95}, a Hilbert scheme argument shows that there exists: 
	 \begin{enumerate}
		 \item a locally topologically trivial family of normal varieties $p : \mathcal V \to B$, where $B$ is a smooth scheme with {\em countably} many components; 
		 \item a morphism $f : \mathcal V \to X_m$, 
	 \end{enumerate}
	 such for any subvariety $V \subset X_m$, there exists $t \in B$ with $f(\mathcal V_t) = V$ . Let $B_{\rm non\, hyp} \subset B$ be the subset parametrizing curves of genus $g \leq 1$. Then, for any irreducible component $V$ of $p^{-1}(B_{\rm non\, hyp})$, the subvariety $\overline{f(V)} \subset X$ admits a dominant family of non-hyperbolic curves, and hence is not of general type. Since $\Omega_X$ is ample, Theorem \ref{thmsubv} implies that ${\rm codim}\, \overline{f(V)} \geq n - 1$ if $f(V) \not\subset (X_m)_{\rm sing}$. The property of $p : \mathcal V \to B$ finally implies that any non-hyperbolic curve  $C \subset X_m$ with $C \not\subset (X_m)_{\rm sing}$ is included in one such $\overline{f(V)}$. This ends the proof.
\end{proof}

We can also prove the following corollary to Theorem \ref{thmsubv}, in the spirit of \cite[Corollary 4]{AA02}.

\begin{corollary}
	Assume that $\Omega_X$ is ample, and let $Y \subset X$ be a closed submanifold. Let $1 \leq l \leq d$ be integers. Assume that $l$ generic points of $Y$ and $d - l$ generic points of $X$ lie on a curve of geometric genus $g$.
	Then if $$l \cdot {\rm codim} Y \leq  \dim X - 2,$$ we have $g > d$.
\end{corollary}

\begin{proof}
	Assume that $g \leq d$. Let $\mathcal C \to V$ be the family of curves and $f : \mathcal C \longrightarrow X$ be the map given by the hypothesis. By the assumption, the image $Z$ of $Y_l \times X_{d-l} \to X_d$ is dominated by the image of $S^d f :S^d \mathcal C \to X_d$. As in \cite{AA02}, we may replace $V$ be a hyperplane section to assume that $S^d f$ is generically finite.

	Since $g \leq d$, the family $\mathcal S^d {\mathcal C} \to V$ is a family of varieties which are not of general type (the fiber over $t$ is a $\mathbb P^{d-g}$-bundle over ${\rm Jac}(C_t)$), and hence $Z$ is not of general type as well. Since $\dim Z = \dim Y_l \times X_{d-l}$, Theorem \ref{thmsubv} implies $\dim(Y_l \times X_{d-l}) < (d-1) \dim X + 2$, hence
	$$
	\dim Y < \frac{1}{l} \left( (l-1) \dim X + 2\right),
	$$
	which gives the result.
\end{proof}

%% file: metric.tex
\subsection{Metric methods}

We will now present a metric point of view on these symmetric products of varieties, which will permit to give several applications to quotients of bounded symmetric domains.

We will use a metric hyperbolicity criterion similar to the one of \cite{cad18}. To express this criterion, we need first to introduce several constants bounding the Ricci curvature on subvarieties of the domain. Let us recall how to define these constants.
\medskip

Let $\Omega$ be a bounded symmetric domain of dimension $n$, and let $h_{\Omega}$ be the Bergman metric on this domain. If $X, Y \in T_{\Omega, x}$ $(x \in \Omega)$, we can define the {\em bisectional curvature} of $h_{\Omega}$ as 
$$B(X, Y) = \frac{i\Theta(h_{\Omega})(X, \overline{X}, Y, \overline{Y})}{||X||_{h_{\Omega}}^2 ||Y||_{h_{\Omega}}^2}.$$

Fix $p \in \mathbb N$. Then, we define 
\begin{equation} \label{eqCp}
	C_{p} = - \max_{X \in T_{\Omega, x}} \max_{V \ni X, \dim V = p} \;\; \sum_{i = 1}^p B(X, e_i),
\end{equation}
where $V \subset T_{\Omega, x}$ runs among the $p$-dimensional subspaces containing $X$, and $(e_i)_{1 \leq i \leq p}$ is any unitary basis of $V$. Since $\Omega$ is homogeneous, this constant does not depend on $x \in \Omega$.

Then, if we normalize the Bergman metric so that $C_{n} = 1$, we have a sequence of positive constants
$$
0 < C_1 \leq C_2 \leq ... \leq C_n = 1.
$$

These constants can be used to state the following criterion for the $p$-hyperbolicity of compactification of a quotient of $\Omega$.

\begin{proposition} [see \cite{cad18}] \label{propcritmetric}
	Let $M$ be a smooth projective manifold, and $D$, $E = \sum_{i} (1 - \alpha_i) E_i$ be $\mathbb Q$-divisors on $X$ such that the support $|E| \cup |D|$ has normal crossings. Let $U = M - (|D| \cup |E|)$, and let $h$ be a smooth K\"ahler metric on $U$, possibly degenerate. Let $p \in \llbracket 1, \dim M \rrbracket$ and let $\alpha > \frac{1}{C_p}$ be a rational number. We make the following assumptions.
	\begin{enumerate}[(i)]
	\itemsep=0em
	\item $h$ is non-degenerate outside an algebraic subset $Z \subset M$, and is modeled after $h_{\Omega}$ on $U - Z$;
	\item the metric induced by $h$ on $\bigwedge^d T_{M}$ has singularities near  any point of $|E_i| - (|D| \cup Z)$ with coefficients of order at most $O(|z|^{2(\alpha_i - 1)})$;
	\item there exists a non-zero section $s$ of $K_U^{\otimes l}$ such that $|| s ||_{(\det h^\ast)^l}^{2/l}$ extends as a continuous function $u$ on $M$, vanishing along $E + D$ at an order strictly larger than $\frac{1}{C_p}$. If $z$ is a local equation for a component of weight $\beta$ in $D + E$, this means that $u = O(|z|^{\frac{\beta}{C_p}(1 + \epsilon)})$ for some $\epsilon > 0$ (recall that $\beta =1$ for the components of $D$, and  $\beta = 1 - \alpha_i$ for the $E_i$).
	\end{enumerate}

	Then, 
	\begin{enumerate}[(a)]
		\item For any subvariety $V \subset M$ with $V \not\subset Z(s) \cup E \cup D \cup Z$ and $\dim V \geq p$, $\dim V$ is of general type. 
		\item For any holomorphic map $f : \mathbb C^p \to M$ with ${\rm Jac}(f)$ generically of maximal rank, we have $f(\mathbb C^p) \subset  Z(s) \cup E \cup D \cup Z$.
	\end{enumerate}
\end{proposition}

\begin{proof}
	The metric $h$ satisfies all the assumptions permitting to apply the proof of Theorem 2 and Theorem 8 of \cite{cad18}. 
	Let us recall that the technique of this proof consists in forming the metric $\widetilde{h} = || s ||_{(\det h^\ast)^m}^{2 \beta} h$ for an adequate $\beta > 0$. We then check that $\widetilde{h}$ induces a positively curved singular metric on the canonical bundle of a desingularization of any subvariety $V$ as in the hypotheses. In the case of a map $f : \mathbb C^p \to M$, we apply the Ahlfors-Schwarz lemma (see \cite{dem12a}) to this metric to obtain a contradiction if $f(\mathbb C^p) \not\subset Z(s) \cup E \cup D \cup Z$.
\end{proof}

\begin{remark} \label{remcritmet} Assume that $X = \quotient{\Gamma}{\Omega}$ is a quotient by an arithmetic lattice, and let $q : M \to \overline{X}^{BB}$ be a log-resolution of the singularities of the Baily-Borel compactification of $X$. Let $U \subset X$ be the smooth locus, and $E_i$ (resp. $D_j$) be the components of the exceptional divisor whose projection intersects $X_{\rm sing}$ (resp. whose projection lies in $\overline{X}^{BB} \setminus X$). For each $i$, let $x_i$ be a generic point of the projection of $E_i$ on $\overline{X}^{BB}$. Let $H_i \subset \Gamma$ be the isotropy group of $x_i$, and let $\alpha_i$ be such that the action of $H_{i}$ on $\Omega$ satisfies the condition $(I'_{x, d, \alpha_i})$ of Section \ref{sectnot}. We associate the multiplicity $\alpha_i$ to $E_i$ by putting $E = \sum_i (1 - \alpha_i) E_i$. We also let $D = \sum_i D_i$.
	
	With these notations, as explained in \cite{cad18}, the hypotheses (i) and (ii) of Proposition \ref{propcritmetric} are satisfied. The condition (iii) is implied by the following more algebraic condition.	

	{ \em (iii')} For $\alpha \in \mathbb Q^\ast_+$, let $L_{\alpha} = q^\ast K_{\overline{X}^{BB}} \otimes \mathcal O( - \alpha (D + E)) $. Then $L_{\alpha}$ is effective for some $\alpha > \frac{1}{C_p}$. 

	Moreover, $Z(s)$ in $(a)$ and $(b)$ can then be replaced by the stable base locus $\mathbb B(L_\alpha)$.
\end{remark}

\begin{remark} \label{remCpsing} We can generalize the conclusion {\em (b)} of Proposition \ref{propcritmetric} to the following situation. Assume that there exists a proper birational holomorphic map $q : M \to M_0$, where $M_0$ is a possibly singular complex variety. Then, under the assumption of the theorem, we can state the following:
	\medskip

	{\em (b')} Let $W = q( Z(s) \cup E \cup D \cup Z)$. Then for any holomorphic map $f : \mathbb C^p \to M_0$ with ${\rm Jac}(f)$ generically of maximal rank, we have $f(\mathbb C^p) \subset W \cup (M_0)_{\rm sing}$.
	\medskip

	To prove this statement, assume by contradiction that there exists a $f : \mathbb C^p \to M_0$ that fails to satisfy the conclusion of {\em (b')}. Let $C$ be a resolution of singularities of the main component of the fiber product $\mathbb C^p \times_{(f, q)} M$. Then, there exists a proper morphism $g : C \to \mathbb C^p$, birational outside a locally finite union of analytic subvarieties of $\mathbb C^p$, and there exists a natural map $h : C \to M$, generically non-degenerate, whose image intersects $U \setminus (Z \cup Z(s) \cup E)$.  Construct $\widetilde{h}$ is as the proof of Proposition \ref{propcritmetric}. Then, the metric $g^\ast \widetilde{h}$ on $C$ is subject to the following version of the Ahlfors-Schwarz lemma.
\end{remark}

\begin{lemma}
	Let $g : C \to \mathbb C^p$ be a proper holomorphic map, realizing an isomorphism outside a countable union of analytic subvarieties of $\mathbb C^p$. Then $T_{C}$ cannot admit any singular metric $h$, with $\det h$ everywhere locally bounded, smooth on a dense open Zariski subset $U$, and satisfying the following inequality on $U$:
	\begin{equation} \label{eqAS}
		d d^c \log \det h \geq \epsilon \omega_h \hspace{2em} (\epsilon > 0).
	\end{equation}
\end{lemma}
\begin{proof}
	Assume by contraction that there exists such a metric. We may assume that $g$ is an isomorphism on some open subset $V \subset C$ containing $U$. We may then see $h$ as a metric on $V \subset \mathbb C^p$, satisfying \eqref{eqAS} on $U$. As $\det h$ is everywhere locally bounded on $V$, and since $dd^c \log \det h \geq 0$ on $U \subset V$, the function $\log \det h$ is psh on $V$. Besides, as $\mathbb C^p$ is normal, we have ${\rm codim}(\mathbb C^p \setminus V) \geq 2$, so $\log \det h$ extends {\em to the whole $\mathbb C^p$} as a psh function, satisfying \eqref{eqAS} in the sense of {\em currents}. This case is however ruled out by the standard Ahlfors-Schwarz lemma stated in \cite{dem12a}.
\end{proof}

\medskip
Our plan is to use the previous proposition in the case where $X$ is a resolution of singularities of a symmetric product of a quotient of a bounded symmetric domain. To do so, we will need some estimates on the $C_p$ when the domain is of the form $\Omega^m$ ($m \in \mathbb N$). The case $p = 1$ is fairly easy to settle: in this case, $- C_1$ is just the maximum of the holomorphic sectional curvature, and we have the following well-known result.

\begin{proposition} Let $\Omega$ be a bounded symmetric domain, and denote by $-\gamma$ the maximum of the holomorphic sectional curvature on $\Omega$. Then we have $$C_1(\Omega^m) = \frac{1}{m} C_1(\Omega) = \frac{\gamma}{m}.$$
\end{proposition}
This can be checked directly by writing the formula for the bisectional curvature of $\Omega^m$, or by remarking that by the polydisk theorem (see \cite{mok89}), it suffices to deal with the case where $\Omega = \Delta^n$. In this case the holomorphic sectional curvature is maximal in the direction of the long diagonals, and the formula can be easily derived.
\medskip

We can use now use this result to study the case of ramified coverings of \emph{smooth compact} quotients of bounded symmetric domains. 

\begin{proposition} \label{prophyp}
	Let $Y = \quotient{\Gamma}{\Omega}$ be a smooth compact quotient, let $p : X \longrightarrow Y$ be a Block-Gieseker covering, and let $\delta = \frac{s}{r}$ be a positive rational number such that be such that $p^\ast K_Y^{\otimes r} = A^{\otimes s}$ for some very ample line bundle $A$. Let $W \subset X$ be the locus where $p$ is non-étale.

	Then if $m \in \mathbb N$ is such that
	$$ \gamma \, \delta > 2 m (m -1),$$
	the variety $X_m$ is Brody hyperbolic modulo $\mathfrak{d}_1(W)$.
\end{proposition}

\begin{proof}
	Let $q :M \to X_m$ be a log-resolution of singularities, let $E \subset M$ be the exceptional locus, and $Z$ be the preimage of $\mathfrak{d}_1(W)$. Let $h_Y$ be the pullback of the Bergman metric on $Y$. This metric is smooth on $Y$, and non degenerate on $Y - W$. This metric induces in turn a natural metric on the smooth locus of $Y_m$, and by pullback, a smooth metric $h$ on $M - E$.

	Let us check that the conditions of Proposition \ref{propcritmetric} are satisfied for $p=1$. Since $h_Y$ is non-degenerate and modeled on $h_\Omega$ on $X - W$, the metric $h$ is non-degenerate and modeled on $h_{\Omega^m}$ on $M - (E \cup Z)$, so the condition (i) is satisfied. 

	It follows directly from the discussion of Section \ref{sectcrit} that the condition $(I'_{x, 1, 1})$ is satisfied for every $x \in X^m$. Hence, the condition (ii) holds for $E = \sum_i E_i$. 
	\medskip

	Let $x \in M - E$. By Proposition \ref{propbaselocus}, for some $N \in \mathbb N$, there exists a section $\sigma$ of $q^\ast A_{\flat}^{\otimes s N} \otimes \mathcal ( - \frac{Ns}{2(m-1)} |E|)$ that does not vanish at $x$. By hypothesis, the line bundles $(A_\flat)^{\otimes s}|_{X_m^{\rm reg}}$ and $K_{X_m^{\rm reg}}^{\otimes r}$ coincide. Thus, if $N$ is divisible enough, $\sigma$ can be seen as a section of the line bundle $( q^\ast K_{X_m} \otimes \mathcal O( - \frac{\delta }{2(m-1)} E) )^{\otimes rN}$. Finally, the holomorphic sections of $q^\ast K_{X_m}^{\otimes r N}$ have bounded norm for the norm induced by $h$, which shows that (iii) is satisfied if $\delta > \frac{2(m-1)}{C_1(\Omega^m)} = \frac{2m(m-1)}{\gamma}$. This is precisely our hypothesis. Moreover, since $x \in M - E$ is arbitrary, the locus cut out by the sections $\sigma$ is included in $M - E$. The conclusion follows as announced from Proposition \ref{propcritmetric}.
\end{proof}

The following result of Hwang-To can be used to give a more explicit constant $\delta$ in the proposition above.

\begin{theorem}[\cite{HT00}] For any smooth compact quotient of a bounded domain $X$, there exists a finite étale cover $X'$ such that $2 K_{X'}$ is very ample.
\end{theorem}

This gives immediately the following series of examples.

\begin{example}
	Let $Y_0 = \quotient{\Gamma}{\Omega}$ be a smooth compact quotient, and let $Y_1 \longrightarrow Y_0$ be the étale cover provided by \cite{HT00}. Let $m \in \mathbb N^\ast$, and let $q$ be an integer such that $q > 4 \frac{m (m-1)}{\gamma}$. 
	
	Now let $X \overset{p}{\longrightarrow} Y_1$ be a Bloch-Gieseker covering such that $p^\ast(K_{Y_1}^{\otimes 2}) = A^{\otimes q}$, with $A$ very ample. Then, we have $\delta \gamma = \frac{q \gamma}{2} > 2 m (m - 1)$, so that $X_m$ is Brody hyperbolic modulo $\mathfrak{d}_1({\rm Sing}(p))$.
\end{example}

\medskip

\begin{example} For $1 \leq i \leq n$, let $X_i$ be a smooth projective curve of genus $g \geq 2$, and fix some integer $q$. For all $i$, since $3 K_{X_i}$ is very ample, we can perform a $q$-fold Bloch-Gieseker covering $p_i : X_i' \longrightarrow X_i$, so that $p_i^\ast (3 K_{X_i}) = A_i^{\otimes q}$, with $A_i$ very ample on $X_i'$. 
\medskip

	Letting $X = X_1' \times ... \times X_n' \overset{p}{\longrightarrow} X_1 \times ... \times X_n = Y$, we have then $p^\ast K_Y^{\otimes 3} = A^{\otimes q}$, where $A = \bigotimes\limits_{1 \leq j \leq n} p_j^\ast K_{X_j}$ is very ample on $X$. The manifold $Y$ is a smooth compact quotient of $\Delta^n$, and $\gamma = \frac{1}{n}$ for this domain. Proposition \ref{prophyp} shows then $X_m$ is Brody hyperbolic modulo $(X_m)_{\mathrm{sing}}$ as soon as 
	$$q \geq 6  m (m - 1) n.$$
\end{example}

\subsection{Non-compact ball quotients} In the case where the domain is the ball, it is possible to give explicit values for the constants $C_p$. The result can be stated as follows when $\dim \Omega \geq 5$.
\medskip

\begin{proposition} \label{propCpball}
	We let $\Omega = \mathbb B^n$ for some $n \geq 5$. Let $m \in \mathbb N$, and fix $p \in \llbracket 1, mn \rrbracket$. Let $k \in \mathbb N$ (resp. $d \in \llbracket 0, n-1\rrbracket$) be the quotient (resp. the remainder) in the euclidean division of $p-1$ by $n$. Then the value of $C_p(\Omega^m)$ is given by the table of Figure \ref{tablecp}. 
\end{proposition}

\begin{figure}[!h]
\centering
\setlength{\tabcolsep}{5pt}
\renewcommand{\arraystretch}{1.5}
\begin{tabular}{|c|c|c|c|c|c|}
\hline
 & $m-k = 1$ & $m-k=2$ & $m-k = 3$ & $m-k = 4$ & $m- k \geq 5$ \\
\hline
	$d =0$ & \multirow{4}{*}{$\frac{d + 2}{n+1}$} & \multicolumn{4}{c|}{$\frac{2}{(m-k)(n+1)}$} \\
\cline{1-1}\cline{3-6}
	$d =1$ &  & $\frac{23}{16} \frac{1}{n+1}$ & $\frac{11}{12} \frac{1}{n+1}$ & $\frac{21}{32} \frac{1}{n+1}$ & \\
\cline{1-1}\cline{3-5}
	$d = 2$ & & $\frac{7}{4} \frac{1}{n+1}$ & \multicolumn{3}{c|}{} \\
\cline{1-1}\cline{3-3}
	$d = 3$ & & $\frac{31}{16} \frac{1}{n+1}$ & \multicolumn{3}{c|}{$\frac{2}{m-k-1} \frac{1}{n+1}$} \\
\cline{1-1}\cline{3-3}
$d \geq 4$ & & \multicolumn{4}{c|}{} \\
\hline
\end{tabular}
	\caption{Values of $C_p$ for the domain $(\mathbb B^n)^m$} \label{tablecp}
\end{figure}

Note the similarity with the case where $\Omega$ is the Siegel upper half-space (see \cite[Proposition 1.4]{cad18}). We will prove Proposition \ref{propCpball} in Section \ref{sectconstantsball}. As an application, we can derive the following hyperbolicity result for symmetric products of ball quotients.  

\begin{corollary} Let $X = \quotient{\Gamma}{\mathbb B^n}$ be a ball quotient by a torsion free lattice with only unipotent parabolic elements, and let $\overline{X} = X \cup D$ be a smooth minimal compactification (see \cite{mok12}). Let $m \geq 1$. Then :
	\begin{enumerate}[(a)]
		\item Let $V \subset \overline{X}_m$ be a subvariety with ${\rm codim} V \leq n - 6$ and $V \not\subset \mathfrak{d}_1(D) \cup (\overline{X}_m)_{\rm sing}$.  Then $V$ is of general type.

		\item Let $p \geq n(m-1) + 6$, and $f : \mathbb C^p \to \overline{X}_m$ be a holomorphic map such that $f(\mathbb C^p) \not\subset  \mathfrak{d}_1(D) \cup (\overline{X}_m)_{\rm sing}$. Then ${\rm Jac}(f)$ is identically degenerate. 
	\end{enumerate}	
\end{corollary}

\begin{proof}
	Let $q : \widetilde{X} \to \overline{X}_m$ be a resolution of singularities. We may assume that $F = q^{-1}(\mathfrak{d}_1(D) \cup (\overline{X}_m)_{\rm sing})$ is a simple normal crossing divisor. Let $\widetilde{D}$ denote the sum of components of $F$ that project in $ \mathfrak{d}_1(D)$, and $E$ the sum of all other components.	

	Let $p \geq n(m - 1) + 6$ be an integer. By Proposition \ref{propCpball}, since $p \geq n(m-1) + 6$, the constant $C_p$ is given by the first column of Figure \ref{tablecp}, and $C_p = \frac{p - n(m-1) + 1}{n+1} > \frac{2\pi}{n+1}$.

	Let $h$ be the metric induced on $U = \widetilde{X} \setminus (E + D)$. Let us check that the assumptions of Proposition \ref{propcritmetric} are satisfied, with $\Omega = (\mathbb B^n)^m$. (i) is obvious, taking $Z = \emptyset$. By Lemma \ref{lemcrit}, since $p \geq n(m-1) + 2$, the condition $(I_{x, p})$ is satisfied above any singular point of $\overline{X}_m$, so Remark \ref{remcritmet} implies that the hypothesis (ii) is satisfied with $\alpha_i = 1$ for any component $E_i \subset E$.  

	To prove (iii), we make use of \cite{BT18}, whose main result shows that the line bundle $K_{\overline{X}} + (1 - \alpha) D$ is ample for any $\alpha > \frac{n+1}{2\pi}$. Let $\alpha \in ] \frac{1}{C_p}, \frac{n+1}{2\pi}[$. Thus, for $l \in \mathbb N$ large enough, and any $x = (x_1, ..., x_m) \in \overline{X}^m \setminus \cup_{i = 1} {\rm pr}_i^{-1}(D)$, we can find a section $\sigma$ of $l \left( K_{\overline{X}} + (1 - \alpha) D\right)$,  such that $\sigma(x_i) \neq 0$ ($1 \leq i \leq m)$. Let $s^\sharp = \bigotimes_{1 \leq j \leq m} {\rm pr}_j^\ast \sigma$. This is a $\mathfrak{S}_m$-invariant section of $K_{X^m}^{\otimes l}$, which descends to a section $s$ of $K_{U}^{\otimes l}$. Let $u = || s ||^{2/l}_{(\det h^\ast)^l}$.
	
	We need to check the conditions on the growth of $u$ near $E + \widetilde{D}$. First, $u$ is bounded near any point of $E$ since $|| s^{\sharp}||_{(\det h_{\Omega}^\ast)^l}$ is continuous on the manifold $X^m$. Besides, by \cite{mum77}, the determinant of the Bergman metric on $K_{\overline{X}} + D$ has logarithmic growth near $D$. Hence, since $\sigma$, seen as a section of $l(K_{\overline{X}} + D)$, vanishes at order $l \alpha$ along $D$, then the function $|| s^{\sharp} ||_{\det h_{\Omega}^\ast}^{2} = \prod_i {\rm pr}_i^\ast ||s||_{h_{\mathbb B^n}}$ vanishes at any order $< l \alpha$ near ${\rm pr_i}^{\ast} D$. Now  $|| s^\sharp||_{(\det h_{\Omega}^\ast)^l}^{2/l} = u \circ \pi$, where $\pi : \overline{X}^m \to \overline{X}_m$ is the projection, so  $u$ vanishes at order $\alpha$ near any point of $\widetilde{D} \setminus E$. As $\alpha > \frac{1}{C_p}$, the section $s$ satisfies the condition (iii). 
	
	Finally, since $x$ was arbitrary outside $\bigcup_{1 \leq i \leq m} {\rm pr}_i^\ast D$, we conclude from Proposition \ref{propcritmetric} that all $p$-dimensional varieties $V \subset \widetilde{X}$, not included in  $E + \widetilde{D}$, are of general type. This proves (a).

	The proof of $(b)$ follows from the conclusion $(b')$ in Remark \ref{remCpsing}, applied with $M = \widetilde{X}$, and $M_0 = \overline{X}_m$.	
\end{proof}

%% file: constants.tex
\subsection{Computation of the curvature constants for the domain $(\mathbb B^n)^m$} \label{sectconstantsball}

We now prove Proposition \ref{propCpball}. We will proceed as in \cite{cad18}, and introduce a certain combinatorial functional whose minimum will give us the value of $C_p (\Omega^m)$.

\begin{defi} \label{defiF} Let 
	$$
	\Delta_m = \{ (r_1, ..., r_m) \in (\mathbb R_+)^m \; | \sum_{1 \leq j \leq m} r_j = 1 \; \text{and}\; r_1 \geq r_2 \geq ... \geq r_m \}.
	$$
	Let $\underline{r} = (r_1, ..., r_m) \in \Delta_m$ and $\Gamma \subset \llbracket 1, m \rrbracket \times \llbracket 1, n \rrbracket$. Denote by $k$ the number of elements of $\Gamma$ in the first column. We assume that $k \leq m - 1$. We define:
	$$
	\mathcal F(\underline{r}, \Gamma) = \left\{ \begin{array}{l}
		2 + \sum\limits_{(i, j) \in \Gamma,\; i \geq 2}  r_i  \; \; \hspace{2em} \text{if} \; k = m - 1 \vspace{1em} \\

		2 \sum\limits_{1 \leq i \leq m} r_i^2 + 2 \sum\limits_{(i, 1) \in \Gamma} r_i + \sum\limits_{(i, j) \in \Gamma, \; j \geq 2} r_i \; \; \hspace{2em} \text{if} \; k \leq m - 2. 
							\end{array} \right.
	$$
\end{defi}
\medskip

From now on, we fix a given minimizer $(\underline{r}, \Gamma)$ for $\mathcal F$, where $\underline{r} \in \Delta_m$, and $\Gamma$ runs among cardinal $p - 1$ subsets of $\llbracket 1, m \rrbracket \times \llbracket 1, n \rrbracket$ with less than $m-1$ elements on the first column. Let $k$ be the number of these elements. We will assume that $(\underline{r}, \Gamma)$ is chosen among all the minimizers so that
\begin{enumerate}[(1)]
	\item $\underline{r} = (r_1, ..., r_m)$ has the maximal number of zero components ;
	\item among all minimizing couples $(\underline{r}, \Gamma)$ satisfying $(1)$, $\Gamma$ is chosen so that {\em $k$ is maximal}.
\end{enumerate}
\medskip

We can make a simple remark on the geometry of $\Gamma$. Let 
$$\Pi = \Gamma \cap \left( \llbracket 1, m \rrbracket \times \llbracket 2, n \rrbracket \right)$$ be the set of elements of $\Gamma$ which are outside of the first column. For each $i \in \llbracket 1, m \rrbracket$, denote by $b_i$ the number of elements of $\Pi$ which are on the $i$-th line. Then, since $r_1 \geq ... \geq r_m$, we see from the formula for $\mathcal F$ that we may suppose that the elements of $\Pi$ are the largest possible in the lexicographic order. This implies that for some $q \in \llbracket 0, m \rrbracket$, $d \in \llbracket 0, n - 2 \rrbracket$, we have $b_{m-j} = n - 1$ ($0 \leq j \leq q - 1$), $b_{m-q} = d$, and $b_{m-j} = 0$ ($m \leq j \leq q + 1$).

\begin{lemma} \label{lemlk} Let $l$ be the maximal integer such that $r_{m - l + 1} = ... = r_m = 0$. We have $l = k$.
\end{lemma}
\begin{proof}
	The proof is exactly the same as the one of \cite[Lemma 3.8]{cad18}, replacing $g$ by $m$, $\Gamma_0$ by $\Gamma$, and "off-diagonal" by "off the first column". 
\end{proof}

The previous proof relies on the following lemma, which will be used frequently in the following.

\begin{lemma} [{see \cite[Lemma 3.9]{cad18}}] \label{lemquadform} Let $a_1 \leq ... \leq a_m$ be non-negative integers, and let $t$ be the smallest integer such that $\sum_{i=1}^t (a_t - a_i) \geq 4$ (let $t = m + 1$ if there is no such integer). Let $\underline{r} \in \Delta_m$ be a minimizer for the quadratic form
	$$
	Q(r_1, ..., r_m) = 2 \sum_{i = 1}^m r_i^2 + \sum_{i = 1}^m a_i r_i.
	$$
	Then $r_{t} = ... = r_m = 0$.
\end{lemma}

We will now compute the several possible values for the minimum $\mathcal F(\underline{r}, \Gamma)$. We will proceed by distinguishing along the value of $k$. There is one simple first case.

\begin{lemma} If $k = m - 1$, then
	$$
	\mathcal F(\underline{r}, \Gamma) = 2 + b_1. 
	$$
\end{lemma}
\begin{proof}
In this case, we have
$$
\mathcal F(\underline{r}, \Gamma) = 2 + \sum_{1 \leq i \leq m} b_i r_i.
$$
	Recall that the $b_i$ are non-decreasing. Since $\underline{r}$ must be an extremum of the function $\mathcal F(\cdot, \Gamma)$, we see that we may chose $\underline{r} = (1, 0, ..., 0)$, which gives the result.
\end{proof}

We will now assume that $k \leq m - 2$, and distinguish several subcases.
\medskip

{\bf Case 0.} $q < k$.

In this situation, since $r_{m-k+1} = ... = r_m = 0$, we simply have $\mathcal F(\underline{r}, \Gamma) = 2 \sum_{i = 1}^{m-k} r_i^2$. The minimum is then reached for $(r_1, ..., r_m) = (\frac{1}{m-k}, ..., \frac{1}{m-k}, 0, ..., 0)$, and the value of the minimum is
$$
\mathcal F(\underline{r}, \Gamma) = \frac{2}{m-k}.
$$
\medskip

{\bf Assumption.} In the remaining cases 1 and 2 below, we will assume that $q \geq k$, which means that $r_{m-q} \neq 0$.
\medskip

{\bf Case 1.} $d \geq 1$. 

 By our previous description of the shape of $\Pi$, this implies that two subcases are {\em a priori} possible.
\medskip

\hspace{1em} {\bf Case 1a.} $q \geq k + 1$, i.e. the line $\{m-k\} \times \llbracket 2, n-1 \rrbracket$ is included in $\Gamma$.
\medskip

\hspace{1em} {\bf Case 1b.} $q = k$ i.e. the only elements of $\llbracket 1, m - k \rrbracket \times \llbracket 2, n-1 \rrbracket$ in $\Gamma$ are the $d$ last elements of $\{m-k\} \times \llbracket 2, n-1 \rrbracket$.

\begin{lemma} The case 1a. cannot occur.
\end{lemma}
\begin{proof}
	In the case 1a, since $r_{m- k} \neq 0$, Lemma \ref{lemquadform} shows that $\sum_{i \leq m-k} (b_{m-k} - b_i) \leq 3$. Hence, all elements of $\llbracket 1, m - k\rrbracket \times \llbracket 2, n-1 \rrbracket$ are in $\Gamma$, except $\delta$ elements on the first line, with $1 \leq \delta \leq 3$. (If $\delta = 0$, we would have $d = 0$).

	This shows that $b_{1} = n- 1 - \delta$, with $1 \leq \delta \leq 3$, and $b_j = n- 1$ ($2 \leq j \leq m- k$). In this setting, the minimizer $\underline{r}$ is of the form $(x, y, ..., y, 0, ..., 0)$ where $y$ is repeated $m-k-1$ times, and $x + (m-k-1)y = 1$. Let $b = m - k -1$.  

	The minimum then equals
	$$
	\begin{aligned}
		\mathcal F(\underline{r}, \Gamma) = 2 x^2 + 2 b y^2 + (n-1) - \delta x. \\
	\end{aligned}
	$$

	We claim that $b \leq 2$. Indeed, if $b \geq 3$, since $n - 1 \geq 4$, we can remove $4 - \delta$ elements on the first line of $\Gamma$, to get a new set $\Gamma'$. If $\underline{r}' \in \Delta_{m}$ is a minimizer for the functional $\mathcal F(\cdot, \Gamma')$, we have $r_2' = ... = r_m' = 0$ by Lemma \ref{lemquadform}. Since $b \geq 3$, there is enough room on the first column of $\Gamma'$ to add back the $4 - \delta$ elements, which gives a new set $\Gamma''$ with strictly more elements on the first column than $\Gamma$. Now 
	$$\mathcal F(\underline{r}', \Gamma'') = \mathcal F(\underline{r}', \Gamma') \leq \mathcal F(\underline{r}, \Gamma') \leq \mathcal F(\underline{r}, \Gamma).$$ (The first equality comes from the fact the $r_2' = ... = r'_m = 0$, and the inequalities are obvious since all $r_i$ are non-negative). This gives a contradiction with our choice of $(\underline{r}, \Gamma$). 

	The same computation as in \cite[Lemma 3.14]{cad18} shows that the case $b = 1$ is impossible.

	Let us finally exclude the case $b = 2$. In this situation $\underline{r} = (x, y, y, 0, ..., 0)$ minimizes $\mathcal F(\underline{r}, \Gamma) = 2x^2 + 4y^2 + (n-1) - \delta x$,
with the constraint $x + 2y = 1$. We check that the minimum is equal to
	$$
	n - \frac{(2 + \delta)^2}{12}.	
	$$
	Since $b = 2$, there are two elements of $\llbracket 1, m \rrbracket \times \{1 \}$ which are not in $\Gamma$, and we can move two elements of the first row $\Gamma$ to get a new set $\Gamma'$ with $m-1$ elements in the first column. Letting $\underline{r}' = (1, 0, ..., 0)$, we have
	$$
	\begin{aligned}
		\mathcal F(\underline{r}', \Gamma') & = 2 + (n-1) - (\delta + 2)  \\
		& = n - 1 - \delta  \\
		&  < n - \frac{(2 + \delta)^2}{12} = \mathcal F(\underline{r}, \Gamma), 
	\end{aligned}
	$$
	since $\delta \in \left\{ 1, 2, 3 \right\}$. This is a contradiction.
\end{proof}

\begin{lemma}
	In the case 1b, there are only 5 possibilities, which are given in the table of Figure 2.
\end{lemma}
\begin{proof}
	In this case, we have $b_{m-q} = d$, and this is the only non-zero $b_j$ with $j \leq m - l$. By Lemma \ref{lemquadform} again, we have $d(m - k - 1) \leq 3$ since $r_{m-k} \neq 0$. Since $d \neq 0$ and $m-k \geq 2$ in the case under study, this gives only only five possibilities. The corresponding values for the minimum of $\mathcal F(\underline{r}, \Gamma) = 2 \sum_{j=1}^{m-k} r_j^2 + d r_{m-k}$ were computed in \cite[Case 2]{cad18}. 
\end{proof}

\begin{figure}[h]
\centering
\renewcommand{\arraystretch}{1.5}
\setlength{\tabcolsep}{5pt}
\begin{tabular}{|c|c|c|c|c|}
\hline
 &  $m-k=2$ & $m-k = 3$ & $m-k = 4$ \\
\hline
$d =1$ &  $\frac{23}{16}$ & $\frac{11}{12}$ & $\frac{21}{32}$ \\
\hline
$d = 2$  & $\frac{7}{4}$ &  &  \\
\hline
$d = 3$  & $\frac{31}{16}$ &  &  \\
\hline
\end{tabular}
\caption{Possible values of the minimum of $\mathcal F$ in the case 1b}
\end{figure}

There is only one remaining case.
\medskip

{\bf Case 2.} $d = 0$. 

\begin{lemma}
	Case 2 cannot occur unless $\Gamma$ is of the form $\llbracket m - k + 1, m \rrbracket \times \llbracket 1, n \rrbracket$.
	The value of the minimum is then
	$$
	\mathcal F(\underline{r}, \Gamma) = \frac{2}{m-k}.
	$$
\end{lemma}
\begin{proof}
	If $\Gamma$ is not of the prescribed form, we have 
	$$
	\mathcal F (\underline{r}, \Gamma) = 2 \sum_{1 \leq j \leq m- k} r_j^2 + (n-1) \sum\limits_{j = m - q + 1}^{m-k} r_j,
	$$
	with $q < k$. Applying another time Lemma \ref{lemquadform}, since $r_{m-k} \neq 0$, we have $(n-1)(m-q) \leq 3$ for all $t \geq 1$. As we assumed that $n \geq 5$, this implies that $q = m$, i.e. $\Gamma$ contains all the elements which are not on the first column. The minimum is then reached for $\underline{r}$ of the form $\underline{r} = (\frac{1}{m-k}, ..., \frac{1}{m-k}, 0, ..., 0)$ ($1/(m-k)$ repeated $m-k$ times), and its value is
$$
	\mathcal F (\underline{r}, \Gamma) = \frac{2}{m-k} + (n-1).
$$
	However, this is absurd. Indeed, let $\Gamma'$ be obtained from $\Gamma$ by moving elements to its $m - k - 1$ empty slots on the first column (recall that we consider sets with at most $m-1$ elements on the first column). 
	
	If $m-k \geq 3$, we may then assume that $\Gamma'$ has less than $(n-1) - 2$ elements on the first line. Letting $\underline{r}' = (1, 0, ..., 0)$, we get
	$$
	\mathcal F(\underline{r}', \Gamma') \leq 2 + (n-3) < \frac{2}{m-k} + (n-1) = \mathcal F(\underline{r}, \Gamma),  
	$$
	which is a contradiction.

	If $m-k = 2$, we may move one element, and assume that $\Gamma'$ has $n-2$ elements on the first line. Then, letting again $\underline{r}' = (1, 0, ..., 0)$, we get
	$$
	\mathcal F(\underline{r}', \Gamma') = 2 + (n-2) = \frac{2}{m-k} + (n-1) = \mathcal F(\underline{r}, \Gamma).
	$$
This is again a contradiction, since we assumed that $\Gamma$ had the maximal number of elements on the first column.
\end{proof}

Putting everything together, we have proved the following.

\begin{proposition} \label{propCp}
	Let $p \in \llbracket 1, mn \rrbracket$. Let $k = \lfloor \frac{p - 1}{n} \rfloor$, and $d = p - 1 - kn$. Let $(\underline{r}, \Gamma)$ be a minimizer for $\mathcal F$, where $\underline{r} \in \Delta_m$, and $\Gamma \subset \llbracket 1, m \rrbracket \times \llbracket 1, n \rrbracket$ is a cardinal $p - 1$ subset with less that $m-1$ elements on the first column. Then
	\begin{enumerate}
		\item the value of $\mathcal F(\underline{r}, \Gamma)$ is given by the table of Figure \ref{figF} ;
		\item we may choose $(\underline{r}, \Gamma)$ so that the elements of $\Gamma$ in the first column are the $(j, 1)$ with $j \geq m-k+1$, and so that $r_{m-k+1} = ... = r_{m} = 0$. 
	\end{enumerate}
\end{proposition}

\begin{figure}[!h]
\centering
\setlength{\tabcolsep}{5pt}
\renewcommand{\arraystretch}{1.5}
\begin{tabular}{|c|c|c|c|c|c|}
\hline
 & $m-k = 1$ & $m-k=2$ & $m-k = 3$ & $m-k = 4$ & $m- k \geq 5$ \\
\hline
$r =0$ & \multirow{4}{*}{$d + 2$} & \multicolumn{4}{c|}{$\frac{2}{m-k}$} \\
\cline{1-1}\cline{3-6}
$d =1$ &  & $\frac{23}{16}$ & $\frac{11}{12}$ & $\frac{21}{32}$ & \\
\cline{1-1}\cline{3-5}
$d = 2$ & & $\frac{7}{4}$ & \multicolumn{3}{c|}{} \\
\cline{1-1}\cline{3-3}
$r = 3$ & & $\frac{31}{16}$ & \multicolumn{3}{c|}{$\frac{2}{m-k-1}$} \\
\cline{1-1}\cline{3-3}
$d \geq 4$ & & \multicolumn{4}{c|}{} \\
\hline
\end{tabular}
\caption{Values of the maxima of $\mathcal F$} \label{figF}
\end{figure}

We will now show that the previously computed maxima permit to give the constant $C_p$. Let us recall how this constant can be computed. 
\medskip

In the following, if $\Omega$ is a bounded symmetric domain, and $X$ is a vector tangent to $\Omega$, we will denote by $B_0^{\Omega}(X, \cdot)$ the following bilinear form:
$$
B_0^\Omega (X, \cdot) : Y \longmapsto i \Theta(h_{\Omega})(X, \overline{X}, Y, \overline{Y}).
$$

	Let $X \in T_{\Omega, 0} $ be a unitary vector. Let $V \subset T_{\Omega^m, 0}$  be a $d$-dimensional vector space containing $X$. We now assume that the pair $(X, V)$ realizes the maximum of \eqref{eqCp}. We let $\mathrm{Aut}(\mathbb B^n)^m$ act on $\Omega$ so that $X$ decomposes in the direct sum $T_{\Omega, 0} = (T_{\mathbb B^n, 0})^{\oplus m}$ as $X = (\alpha_1 e_1^1, ..., \alpha_m e_1^m)$, where $(e_1^i, ..., e_n^i)$ denotes a unitary basis of the $i$-th factor $T_{\mathbb B^n}$. We let $r_i = \alpha_i^2$ ($1 \leq i \leq m$), so that $\sum_{1 \leq i \leq m} r_i = 1$. We may assume that $r_1 \geq r_2 \geq ... \geq r_m$.

	By our choice of $(X, V)$, we have
	\begin{equation} \label{exprcp}
		C_{p} = - B_0(X, X) + \sum_{\lambda \in S(V)} \lambda,
	\end{equation}
	where $S(V)$ is the set of the $p-1$ eigenvalues of the restriction of the hermitian form $ - B_0(X, \cdot)$ to $X^{\perp} \cap V$ (with multiplicities). We let $W \subset V$ be a $(p-1)$-dimensional vector subspace, spanned by corresponding eigenvectors, so that $V = \mathbb C X \overset{\perp}{\oplus} W$.
	\medskip

	Let us now explain how to compute the eigenvalues of the hermitian form $B_0^{\Omega}(X, \cdot)$ on the space $T_{\Omega, 0}$. First, it is easy to show that for $U = (U_1, ..., U_m)$, $V = (V_1, ..., V_m)$ in $T_{\Omega, 0}$, we have
	$$
	B_0^{\Omega} (U, V) = \sum_{1 \leq m} B_0^{\mathbb B^n}(U_i, V_i).
	$$

To simplify the computation, we will temporarily adopt a new normalization on $h_{\mathbb B^n}$, so that for any $U \in T_{\mathbb B^n, 0}$, the eigenspaces of $- B_0^{\mathbb B^n}(U, \cdot)$ are
$$
\left\{ \begin{array}{ll}
	\mathbb C \cdot U & \text{for the eigenvalue} \; 2 || U ||^2 ; \\
	U^{\perp} \subset T_{\mathbb B^n} & \text{for the eigenvalue} \; || U||^2. \\
\end{array} \right.
$$

Thus, with this normalization, the eigenvalues of $B_0^{\Omega} (X, \cdot)$ are $2 r_i$ (with multiplicity $1$, and eigenvector $e_1^i$) and $r_i$ (with multiplicity $n-1$, with eigenvectors $e_2^i, ..., e_n^i$), for $1 \leq i \leq m$.

\begin{proposition}
	With the above normalization, the constant $C_p$ is equal to the minimum of $\mathcal F$.
\end{proposition}

	The proof is the same as in \cite{cad18}, so we will only sketch it briefly.
\begin{lemma}
	We have $C_p \geq \min\limits_{\underline{r}, \Gamma} \mathcal F(\underline{r}, \Gamma)$, where $\underline{r} \in \Delta_m$, and $\Gamma \subset \llbracket 1, m \rrbracket \times \llbracket 1, n \rrbracket$ runs among the cardinal $p-1$ subsets with less that $m-1$ elements on the first column.
\end{lemma}
\begin{proof}

	We can decompose $W = W_1 \overset{\perp}{\oplus} W_2$, where 
	$$W_1 \subset \overset{\perp}{\bigoplus\limits_{1 \leq i \leq m}} \mathbb C e_1^i,\; \text{and} \; W_2 \subset \overset{\perp}{\bigoplus\limits_{1 \leq i \leq m}} {\rm Vect}(e_2^i, ..., e_n^i).$$ Let $k = \dim W_1$. By the description above of the eigenvalues of $B_0^{\Omega}(X, \cdot)$, we see that $W_2$ is spanned by $p-1 -k$ eigenvectors corresponding to the eigenvalues $r_i$ ($1 \leq i \leq m$).

	Let $S_1$ be the sum of the $k$ smallest of the $2r_i$, and $S_2$ be the sum of the $k$-th smallest of the eigenvalues of $-B_0(X, \cdot)$ on $W_2$. Then
	$$
	\begin{aligned}
		C_p & = -B_0(X, X) - {\rm Tr} B_0(X, \cdot)|_{W_1} -  {\rm Tr} B_0(X, \cdot)|_{W_2} \\
		& \geq - B_0(X, X) + S_1 + S_2 = 2 \sum_{i \geq i} r_i^2 + S_1 + S_2.
	\end{aligned}
	$$
	The eigenvalues appearing in $S_1$ and $S_2$ can be indexed by a subset $\Gamma \subset \llbracket 1, m \rrbracket \times \llbracket 1, n \rrbracket$, with $k$-elements of the first column corresponding to the $k$-th smallest $2 r_i$, and the elements $(i, j)$ to the $r_j$ if $j \geq 2$. 
	
	There are two cases to distinguish. First, if $k \leq m -1$, what has just been said shows that $C_p \geq \mathcal F(\underline{r}, \Gamma)$.
	
	Now, if $k = m-1$, then $\mathbb C X \overset{\perp}{\oplus} W_1 = \bigoplus_{i=1}^m \mathbb C \cdot e_1^i$, so
	$$ \begin{aligned}
		-B_0(X, X)   - {\rm Tr} B_0(X, \cdot)|_{W_1} & = {\rm Tr} \left( -B_0(X, \cdot)|_{\bigoplus_{i=1}^m \mathbb C \cdot e_1^i} \right) \\
		& = 2.
	\end{aligned}$$ $C_p$ is equal to the first case of the definition of $\mathcal F$ in Definition \ref{defiF}, so $C_p = \mathcal F(\underline{r}, \Gamma)$. 
\end{proof}

\begin{lemma} We have $\min_{\underline{r}, \Gamma} F(\underline{r}, \Gamma) \geq C_p$. 
\end{lemma}
\begin{proof}
	Let $\underline{r}$ and $\Gamma$ realizing this minimum. Let $W$ be the $p-1$-dimensional space spanned by the eigenvectors corresponding to the elements of $\Gamma$, and let $X = (\sqrt{r_1} e_1^1, ..., \sqrt{r_m} e_1^m)$. By Proposition \ref{propCp} (2), we see that $W \subset X^{\perp}$, so if we let $V = \mathbb C \oplus W$, we have
	$$
	\begin{aligned}
		- {\rm Tr} B_0(X, \cdot)|_{V} & = - B_0(X, X) - {\rm Tr}\, B_0(X, \cdot)|_{W} \\
		& =  \mathcal F(\underline{r}, \Gamma).
	\end{aligned}
	$$
	As $C_p$ is defined to be the minimum of the left hand side for all $X$ and $V$ with $\dim V= p$ and $X \in V$ unitary, this shows that $\mathcal F(\underline{r}, \Gamma) \geq C_p$.
\end{proof}

Thus, the table \ref{figF} gives the constants $C_p$ with our simplifying normalization. To obtain the table \ref{tablecp}, for which the normalization is chosen so that $C_{nm} = 1$, we must replace $C_p$ by $\frac{C_p}{C_{nm}}$. In our current normalization, we have $C_{nm} = n + 1$ according the the first column of table \ref{figF}. This ends the proof of Proposition \ref{propCpball}.